\documentclass[11pt,a4paper]{amsart}
\usepackage[utf8]{inputenc}
\usepackage[T1]{fontenc}
\usepackage{amsmath}
\usepackage{amssymb}
\usepackage{stmaryrd}
\usepackage{graphicx}
\usepackage{setspace}
\usepackage{soul}
\usepackage{latexsym}
\usepackage{mathrsfs}
\usepackage{fancyhdr}
\usepackage{enumerate}
\usepackage[english]{babel}
\usepackage{etoolbox}
\usepackage{colonequals}
\usepackage[colorlinks=true,linkcolor=Maroon,citecolor=LimeGreen,pagebackref,linktocpage=true]{hyperref}
\usepackage[usenames,dvipsnames]{xcolor}
\usepackage{float}
\usepackage{cases}
\usepackage{amsfonts}
\usepackage{geometry}
\usepackage{tikz}
\usepackage[all]{xy}
\usepackage{tikz-cd}
\usepackage{multicol}

\numberwithin{equation}{section}

\newcommand{\bA}{\mathbb{A}}

\newcommand{\bZ}{\mathbb{Z}}
\newcommand{\cA}{\mathcal{A}}
\newcommand{\cB}{\mathcal{B}}

\newcommand{\ve}{\mathcal{E}}
\newcommand{\vf}{\varphi}

\newcommand{\wh}{\widehat}

\newcommand{\rep}{\mathrm{rep}}
\newcommand{\Jac}{\mathrm{J}}

\newcommand{\Hom}{\mathrm{Hom}}

\newcommand{\modu}{\mathrm{mod}}

\newcommand{\Gr}{\mathrm{Gr}}

\newcommand{\ra}{{\rangle}}
\newcommand{\la}{{\langle}}

\newtheorem{theorem}{Theorem}[section]

\newtheorem{conjecture}[theorem]{Conjecture}

\newtheorem{corollary}[theorem]{Corollary}

\newtheorem{definition}[theorem]{Definition}
\newtheorem{example}[theorem]{Example}

\newtheorem{lemma}[theorem]{Lemma}

\newtheorem{proposition}[theorem]{proposition}

\theoremstyle{remark}
\newtheorem{remark}[theorem]{Remark}

\title[Finite Dimensional 2-cyclic Jacobian Algebras]{Finite Dimensional 2-cyclic Jacobian Algebras}

\author[Y. Li]{Yiyu Li}
\address{Department of Mathematics, Sichuan University, Chengdu 610064, P.R.China}
\email{liyiyumath@gmail.com}
\author[L. Peng]{Liangang Peng}
\address{Department of Mathematics, Sichuan University, Chengdu 610064, P.R.China}
\email{penglg@scu.edu.cn}


\begin{document}
\begin{sloppypar}
\maketitle

\begin{abstract}
In this paper, we start with a class of quivers that containing only 2-cycles and loops, referred to as 2-cyclic quivers. We prove that there exists a potential on these quivers that ensures the resulting quiver with potential is Jacobian-finite. As an application, we first demonstrate, using covering theory, that a Jacobian-finite potential exists on a class of 2-acyclic quivers. Secondly, by using the 2-cyclic Caldero-Chapoton formula, the $\tau$-rigid modules over the Jacobian algebras of our proven Jacobian-finite 2-cyclic quiver with potential can categorify Paquette-Schiffler's generalized cluster algebras in three specific cases: one for a disk with two marked points and one 3-puncture, one for a sphere with one puncture, one 3-puncture and one orbifold point, and another for a sphere with one puncture and two 3-punctures. \end{abstract}

\section{Introduction}\label{sec: intro} 

Quivers with potentials were introduced by Derksen, Weyman, and Zelevinsky \cite{DWZ1, DWZ2} as a tool in the connection of representation theory and cluster algebras. The potentials are formal linear combinations of cycles in the quiver, which enabled the construction of the Jacobian algebra. These tools from \cite{DWZ1, DWZ2} have addressed most of the questions in cluster algebras involving skew-symmetric matrices, including sign-coherence of $c$-vectors and properties like the $F$-polynomial having a constant term of $1$.

Jacobian-finite quivers with potentials are crucial for providing a finite-dimensional algebraic structure, which makes the related research objects tractable and computationally manageable. Here, we introduce some results based on the Jacobian-finite condition or assumption. In \cite{Amiot09}, Amiot proved that the generalized cluster category associated with a Jacobi-finite quiver with potential is a Hom-finite, Krull-Schmidt, 2-Calabi-Yau category with cluster tilting objects. This finite-dimensional property ensures that the cluster tilting objects in these generalized cluster categories correspond one-to-one with the support $\tau$-tilting pairs of the corresponds finite-dimensional Jacobian algebras, as shown by Adachi, Iyama, and Reiten \cite{AIR14}. 
Many specific types of Jacobian-finite quivers with potentials have been provided so far. In \cite{LF1}, Daniel Labardini-Fragoso introduced the Labardini potentials for the triangulated surfaces and proved the Labardini potentials are Jacobian-finite for the surfaces with boundary and for the once-punctured torus \cite{LF1,LF2}. Extending this research, Sonia Trepode and Yadira Valdivieso-Díaz proved the Jacobian-finite property for Labardini potential on spheres with at least five punctures \cite{TV17}. Finally, Sefi Ladkani completed the picture by proving the Jacobian-finite property for the remaining cases \cite{Lad12}. Additionally, Jie Zhang and Wen Chang defined quivers with potentials associated with any Postnikov diagram and proved that these quivers with potentials are Jacobian-finite \cite{CZ23}.

So far, including the work listed above, most research of Jacobian-finite potentials have focused on quivers that are 2-acyclic. In this paper, after introducing the concepts of quivers with potentials and 2-cyclic quivers in Section \ref{sec: preliminary}, we provide a potential for the 2-cyclic quiver of type $\bA_n$ in Section \ref{section:jacobian_finite_2cyclic}. This quiver contains only 2-cycles and loops, with its loop-free part being the double quiver of the $\bA_n$ quiver. Under the $[1,m]$-finite dimensional condition, we prove that this 2-cyclic quiver with potential is Jacobian-finite by showing that all elements of the 2-cyclic Jacobian algebras are linear combinations of paths with a bounded length. As an application, in Section \ref{sec: Application}, we prove that there is a 2-acyclic quiver induced from the $\bZ_3$-cover of this 2-cyclic quiver with potential, which is also Jacobian-finite. For the 2-cyclic quiver $\overline{\bA_2}^{\emptyset}$, $\overline{\bA_2}^{[1,2]}$ and $\overline{\bA_3}^\emptyset$ since the cluster category of these 2-cyclic quivers with potentials are Hom-finite, and Paquette-Schiffler's generalized cluster algebras can be defined on these 2-cyclic quivers, we show that, through the 2-cyclic Caldero Chapoton map defined on the module categories of the 2-cyclic Jacobian algebras in Section \ref{sec: Categorification_of_PS_Generalized_Cluster_Algebras}, the indecomposable $\tau$-rigid modules of our resulting finite-dimensional 2-cyclic Jacobian algebras correspond to the cluster variables of Paquette-Schiffler's generalized cluster algebras.

In this paper, \(K\) is an algebraic closed field. For $a,b\in \mathbb Z$, we denote \([a,b] := \{k \in \mathbb{Z} \mid a \leq k \leq b\}\). Additionally, when \(a > b\), we interpret them as empty sets. we always discuss algebras that are unital and associative and categories are essentially small and $K$-linear.




\section{Preliminaries}\label{sec: preliminary}
\subsection{Quivers and Path algebras}
A \textbf{quiver} $Q$ is an oriented graph consisting of a set of vertices $Q_0$ and a set of arrows $Q_1$. Each arrow $a \in Q_1$ has a source vertex $s(a)$ and a target vertex $t(a)$. A \textbf{path} in a quiver is a sequence of arrows $a_1 a_2 \cdots a_n$ such that $t(a_i) = s(a_{i+1})$ for $i = 1, \ldots, n-1$. The \textbf{length} of a path is the number of arrows in the sequence. A \textbf{cycle} is a path that starts and ends at the same vertex. We call a quiver 2-acyclic if the length of every cycle in $Q$ are at least 3. A \textbf{double quiver} $\overline{Q}$ of a quiver $Q$ is $\overline{Q}=(Q_0,Q_1\cup Q_1^\star,s,t)$, where $Q^\star_1=\{a^\star|a\in Q_1\}$ and $s(a)=t(a^\star),t(a)=s(a^\star)$. 

The \textbf{path algebra} $KQ$ of a quiver $Q$ over a field $K$ is the vector space over $K$ with a basis consisting of all paths in $Q$, including a trivial path $e_i$ at each vertex $i \in Q_0$ with $s(e_i)=t(e_i)=i$. Multiplication in $K Q$ is defined by composition of paths: if $p = a_1 a_2 \cdots a_m$ and $q = b_1 b_2 \cdots b_n$ are paths with $t(a_m) = s(b_1)$, then $pq = a_1 a_2 \cdots a_m b_1 b_2 \cdots b_n$; otherwise, the product is $0$. The identity element is the sum of the trivial paths: $\sum_{i \in Q_0} e_i$, the product of the trivial path is defined to be $e_ie_j=\delta_{ij}$. We denote the ideal generated by all the arrows of $KQ$ by $\mathfrak{m}$.

\subsection{Quiver representations and quiver Grassmannians}A \textbf{representation} (or module) $M$ of a quiver $Q$ consists of a family of vector spaces $\{ M_i \}_{i \in Q_0}$ and a family of linear maps $\{ M_a : M_{s(a)} \to M_{t(a)} \}_{a \in Q_1}$, we denote the category of finite dimensional representations of $Q$ by $\rep(Q)$. For each $M\in \rep(Q)$, a representation $N$ is a \textbf{subrepresentation} of $M$ if $N_i$ is a subspace of $M_i$ for each $i \in Q_0$ and $M_a(N_{s(a)}) \subseteq N_{t(a)}$ for all $a \in Q_1$. The \textbf{dimension vector} of a representation $M$ is the vector $\dim M = (d_i)_{i \in Q_0}$, where $d_i = \dim_K M_i$. For a given dimension vector $\mathbf{v} = (v_i)_{i \in Q_0}$, the \textbf{quiver Grassmannian} $\mathrm{Gr}_{\mathbf{v}}(M)$ is the set of subrepresentations $N \subseteq M$ such that $\dim N_i = v_i$ for each $i \in Q_0$. Since quiver Grassmannians are algebraic varieties, their structure allows the use of geometric methods to study their properties.

\subsection{$g$-vectors}Assume that $|Q_0|=n$, for any finite generated module $M\in \rep(Q)$, we define the \textbf{$g$-vector} $g_M$ of $M$ to be the vector $(a_1-b_1,...,a_n-b_n)^T\in \bZ^n$, where $a_i,b_i$ are obtained from the following right exact sequence which is the minimal projective presentation of $M$:
\[
\begin{tikzcd}
\oplus_{i=1}^nP^{b_i}_i \arrow[r] & \oplus_{i=1}^nP^{a_i}_i \arrow[r] & M \arrow[r] & 0,
\end{tikzcd}
\]
where $P_i$ is the indecomposable projective module.

\subsection{Quivers with Potentials}

The \textbf{completed path algebra} $\widehat{KQ}$ is defined to be the direct limit $\varinjlim_n KQ/\mathfrak{m}^n$, which naturally possesses an algebraic structure and an $\mathfrak{m}$-adic topological structure. An element $W \in \widehat{KQ}/\{\widehat{KQ}, \widehat{KQ} \}$ is called a \textbf{potential} on $Q$, where $\{\widehat{KQ}, \widehat{KQ} \}$ represents the closed subspace of $\widehat{KQ}$ generated by all commutators. The pair $(Q, W)$ is called a \textbf{quiver with potential}. Potentials are linear combinations (potentially infinite sums) of cycles in $Q$. Two cycles in $Q$ are identified in $\widehat{KQ}/\{\widehat{KQ}, \widehat{KQ} \}$ if they are the same up to cyclic permutation. For an arrow $a$ in $Q$, a cyclic derivation $\partial_a$ with respect to $a$ is a map:
$$\widehat{KQ}/\{\widehat{KQ}, \widehat{KQ}\} \rightarrow \widehat{KQ},$$
defined by $\partial_a(p) = vu$ for a path $p = uav \in \widehat{KQ}/\{\widehat{KQ}, \widehat{KQ}\}$. For any fixed $(Q, W)$, the Jacobian ideal is defined as the closure $\overline{\langle \partial_a(W) \rangle}_{\forall a \in Q_1}$, and the quotient algebra $~\widehat{KQ} / \overline{\langle \partial_a(W) \rangle}_{\forall a \in Q_1}$ is called the \textbf{(completed) Jacobian algebra}, denoted by $\wh{\Jac}(Q,W)$. A potential $W$ is said to be \textbf{Jacobian-finite} if its corresponding Jacobian algebra is finite dimensional.

There are some related concepts of quivers with potentials:
\begin{definition}\label{def:QP_Related_Concepts}
\item[(1)] Let $(Q',W')$ and $(Q'',W'')$ be two quivers with potentials satisfied $Q'_0=Q''_0$, their direct sum is defined to be the quiver with potential $(Q',W') \oplus (Q'',W'') := (Q,W)$, where $Q = (Q'_0, Q'_1 \cup Q''_1), W = W' + W''$.
\item[(2)] Let $(Q,W)$ and $(Q',W')$ be two quivers with potentials with $Q_0 = Q'_0$. If there exists an algebra isomorphism $\varphi: \widehat{KQ} \rightarrow \widehat{KQ'}$ that induces cyclic equivalence $\varphi(W)=W'$, we say that $(Q,W)$ and $(Q',W')$ are right equivalent, denoted as $(Q,W) \simeq^r (Q',W')$.
\item[(3)] We say that a quiver with potential $(Q,W)$ is trivial if $W$ is the sum of the cycles of length $2$ and $\wh{\Jac}(Q,W)\simeq \bigoplus_{i\in Q_0} Ke_i$. We say that a quiver with potential $(Q,W)$ is reduced if $W$ is the sum of cycles of length at least $3$.
\end{definition}

In \cite{DWZ1}, it has been proven that any quiver with potential is right equivalent to a direct sum of a trivial quiver with potential and a reduced quiver with potential. This direct sum decomposition is unique up to right equivalence. In subsequent discussions, we refer to $(Q_{triv}, W_{triv})$ as the trivial summand and $(Q_{red}, W_{red})$ as the reduced summand of $(Q,W)$.

\subsection{DWZ's mutation of a quiver with potential}
In \cite{DWZ1}, another significant concept is introduced, called the DWZ-mutation of quivers with potentials, which serves as an algebraic analog of mutation in cluster algebras. A vertex $k$ in a quiver with potential $(Q,W)$ is called mutable if $Q$ has no 2-cycles or loops at $k$. In this section, we always assume $k$ is a mutable vertex in $(Q,W)$.

\begin{definition}\label{def:DWZMu}
Let $(Q,W)$ be a quiver with potential, $k\in Q_0$, denote $N(k):=\{a|s(a)=k\text{~or~}t(a)=k\}$, there is a quiver with potential $\widetilde{\mu_k}(Q, W):=(Q', W')$ obtained as follows:
\begin{itemize}
    \item $Q'$ has the same vertex set as $Q$, the arrow set of $Q'$ is given by reversing all the arrows connecting to $k$ and by adding arrows $[ab]$ when $ab$ is a path such that $t(a)=s(b)=k$, we present the arrow set $Q_1'$ as follows:  
    \begin{equation}\label{eq:QP_mutation_arrow_set}
        Q'_1:=Q_1\setminus N(k)\cup \{a^\star|s(a^\star)=t(a),t(a^\star)=s(a),a\in N(k)\}\cup\{[ab]| t(a)=s(b)=k\}.
    \end{equation}
    \item $W' = [W] + \sum_{a \in Q^1: s(a) = k} a^\star b^\star [ba]$, where $[W]$ is obtained by replacing all length-2 paths $ab$ through $k$ in $W$ with $[ab]$.
\end{itemize}
The process of obtaining $(Q', W')$ from $(Q, W)$ is called DWZ-mutation. We denote the reduced part of $\widetilde{\mu_k}(Q, W)$ as $\mu_k(Q, W)$.
\end{definition}

In \cite{DWZ1}, the following results concerning DWZ-mutation are proved:
\begin{theorem}\label{theorem:Main_theorem_DWZ}
For any quiver with potential $(Q, W)$, $k\in Q_0$, we have:
\begin{itemize}
    \item The operation $\mu_k$ acting on the reduced part $(Q_{red}, W_{red})$ is an involution, i.e., $\mu_k^2(Q_{red}, W_{red}) = (Q_{red}, W_{red})$.
    \item $(Q, W)$ is Jacobian-finite if and only if $\widetilde{\mu_k}(Q, W)$ is also Jacobian-finite.
\end{itemize}
\end{theorem}

\begin{definition}\label{def:NonDegDWZ}
Let $i_1, \ldots, i_s$ be a sequence of vertices in $Q$. For any $t \in [1, s]$, if $\mu_{i_t} \cdots \mu_{i_1}(Q, W)$ contains no 2-cycles or loops, we call $(Q, W)$ $(i_1, \ldots, i_s)$-non-degenerated. If $(Q, W)$ is $(i_1, \ldots, i_s)$-non-degenerated for any sequence "$i_1, \ldots, i_s$" of vertices in $Q$, then $(Q, W)$ is called non-degenerated.
\end{definition}

Concerning the existence of non-degenerated potentials for 2-acyclic quivers, \cite{DWZ1} proved the positive results for the existence. However, it should be noted that the proof of this theorem is not constructive. Therefore, for an arbitrary quiver, there is no explicit algorithm to find a non-degenerate potential:

\begin{theorem}\label{theorem:Non-DegPot}\cite{DWZ1}
If $K$ is uncountable, then any quiver $Q$ without 2-cycles or loops has a non-degenerated potential.
\end{theorem}

\subsection{2-cyclic quivers}
We now provide the definition of a 2-cyclic quiver.

\begin{definition}\label{def:2CQ}
    A finite quiver $Q=(Q_0, Q_1, s, t)$ is called 2-cyclic if it satisfies:
\begin{enumerate}
    \item For any arrow $a_i: i \rightarrow j$ where $s(a_i)\neq t(a_i)$, there exists exactly one arrow $b_i: j \rightarrow i$ in the opposite direction.
    \item There is at most one loop at any vertex $i$ in $Q$. If there is a loop at vertex $i$, it is denoted as $\ve_i$.
\end{enumerate}    
\end{definition}

One can easy verify that the loop-free part of a 2-cyclic quiver is a double quiver of a simple-laced quiver. Therefore, in the following discussions, for the sake of clarity in notation, we always denote a 2-cyclic quiver $Q$ as \(\overline{D}^I\) if the loop-free part of $Q$ is the double quiver of a quiver \(D\) and $I=\{i\in Q_0|\ve_i\in Q_1\}$ .

\section{Jacobian-finite 2-cyclic Quivers With Potentials}\label{sec: Main} 
\subsection{2-cyclic Quivers With Potentials}\label{section:jacobian_finite_2cyclic}
Let $(\overline{\bA}_n^{[1,m]},W_n^{[1,m]})$ be a 2-cyclic quiver with potential of the form:
\[
\begin{tikzcd}
\overline{\bA}_n^{[1,m]}:&1 \arrow[r, "a_1", shift left] \arrow["\ve_1"', loop, distance=2em, in=125, out=55] & 2 \arrow[l, "b_1", shift left] \arrow[r, "a_2", shift left] \arrow["\ve_2"', loop, distance=2em, in=125, out=55] & \cdots \arrow[r, "a_{m-1}", shift left] \arrow[l, "b_2", shift left] & m \arrow[l, "b_{m-1}", shift left] \arrow[r, "a_{m}", shift left] \arrow["\ve_m"', loop, distance=2em, in=125, out=55] & m+1 \arrow[r, "a_{m+1}", shift left] \arrow[l, "b_m", shift left] & \cdots \arrow[l, "b_{m+1}", shift left] \arrow[r, "a_{n-1}", shift left]  & n \arrow[l, "b_{n-1}", shift left]
\end{tikzcd}
\]
\begin{equation}
\begin{aligned}
W_n^{[1,m]} &:= \sum_{i=1}^{m} k_i \ve_i^3 + \sum_{i=1}^{n-1} t_i (a_i b_i)^3 + \sum_{i=1}^{m} 3 \ve_i a_i b_i + \sum_{i=2}^{m} 3 \ve_i b_{i-1} a_{i-1} + \sum_{i=m}^{n-2} 3 a_i a_{i+1} b_{i+1} b_i
,\end{aligned}
\label{eq:potential_form}
\end{equation}
with $k_i,t_i\in K$. If $W_n^{[1,m]}$ satisfied $n\equiv m\pmod 2$ and the following condition holds for any $1 \leq i' \leq m$ and $1 \leq s' \leq \frac{n-m-2}{2}$:
\begin{equation}
\begin{aligned}
\left( \sum_{i=i'}^{m} (-1)^i k_i + \sum_{s=1}^{s'} (-1)^s t_{m+2s-1} \right) \neq 0
\end{aligned}
\label{eq:finite_dimensional_condition}
\end{equation}
then we say $W_n^{[1,m]}$ satisfies the $[1,m]$-finite dimensional condition. In this section, we will prove that any quiver with potential $(\overline{\bA}_n^{[1,m]},W_n^{[1,m]})$ that satisfies the $[1,m]$-finite dimensional condition is Jacobian finite.

Consider the Jacobian ideal corresponding to $W_n^{[1,m]}$, $\langle \partial_\alpha W_n^{[1,m]} \rangle_{\alpha \in Q_1}$, whose generators can be classified into the following three types:
\begin{equation}
\begin{aligned}
\frac{1}{3}\partial_{\ve_i} W_n^{[1,m]} &= \begin{cases} 
k_1 \ve_1^2 + a_1 b_1, & i = 1; \\
k_i \ve_i^2 + a_i b_i + b_{i-1} a_{i-1}, & 2 \leq i \leq m;
\end{cases}
\end{aligned}
\label{eq:jacobian_ideal_ve}
\end{equation}

\begin{equation}
\begin{aligned}
\frac{1}{3}\partial_{a_i} W_n^{[1,m]} &= \begin{cases} 
t_i (b_i a_i)^2 b_i + b_i \ve_i + \ve_{i+1} b_i, & 1 \leq i \leq m-1; \\
t_m (b_m a_m)^2 b_m + b_m \ve_m + a_{m+1} b_{m+1} b_m, & i = m; \\
t_i (b_i a_i)^2 b_i + a_{i+1} b_{i+1} b_i + b_i b_{i-1} a_{i-1}, & m+1 \leq i \leq n-2; \\
t_{n-1} (b_{n-1} a_{n-1})^2 b_{n-1} + b_{n-1} b_{n-2} a_{n-2}, & i = n-1;
\end{cases}
\end{aligned}
\label{eq:jacobian_ideal_a}
\end{equation}

\begin{equation}
\begin{aligned}
\frac{1}{3}\partial_{b_i} W_n^{[1,m]} &= \begin{cases} 
t_i a_i (b_i a_i)^2 + \ve_i a_i + a_i \ve_{i+1}, & 1 \leq i \leq m-1; \\
t_m a_m (b_m a_m)^2 + \ve_m a_m + a_m a_{m+1} b_{m+1}, & i = m; \\
t_i a_i (b_i a_i)^2 + a_i a_{i+1} b_{i+1} + b_{i-1} a_{i-1} a_i, & m+1 \leq i \leq n-2; \\
t_{n-1} a_{n-1} (b_{n-1} a_{n-1})^2 + b_{n-2} a_{n-2} a_{n-1}, & i = n-1;
\end{cases}
\end{aligned}
\label{eq:jacobian_ideal_b}
\end{equation}
It is known that in the Jacobian algebra $\widehat{\Jac}(\overline{\bA}_n^{[1,m]},W_n^{[1,m]})$, $\partial_\alpha W_n^{[1,m]} = 0$ for $\alpha \in Q_1$. Therefore, equations \eqref{eq:jacobian_ideal_ve}, \eqref{eq:jacobian_ideal_a}, and \eqref{eq:jacobian_ideal_b} provide three types of relations in $\widehat{\Jac}(\overline{\bA}_n^{[1,m]},W_n^{[1,m]})$. Based on these relations, we will prove some equations in $\widehat{\Jac}(\overline{\bA}_n^{[1,m]},W_n^{[1,m]})$ in the following lemmas, and use these equations to prove that $\widehat{\Jac}(\overline{\bA}_n^{[1,m]},W_n^{[1,m]})$ is a finite-dimensional algebra.
\subsection{Additional Relations in $\widehat{\Jac}(\overline{\bA}_n^{[1,m]},W_n^{[1,m]})$}
\begin{lemma}\label{lemma:jacobian_finite_relations}
In the Jacobian algebra $\widehat{\Jac}(\overline{\bA}_n^{[1,m]},W_n^{[1,m]})$, we have:
\begin{enumerate}
    \item[(I)] For $i \leq m$, $a_1 \cdots a_{i-1} \ve_i = \left( (-1)^{i-1} + f_i(\ve_1) \right) \ve_1 a_1 \cdots a_{i-1}$, where $f_i(x) \in K[x]$ and $x^3 \mid f_i(x)$.
    \item[(II)] For $i \leq m$, $a_1 \cdots a_{i-1}(a_i b_i) a_i = = \left(\left(\sum_{j=1}^i (-1)^{i-j+1} k_j \right) \ve_1^2 + g_i(\ve_1)\right) a_1 \cdots a_i$, where $g_i(x) \in K[x]$ and $x^4 \mid g_i(x)$.
    \item[(III)] For $i > m$, $a_1 \cdots a_{i-1} (a_i b_i)^2 a_i = \left(r_i \ve_1^2 + g_i(\ve_1)\right) a_1 \cdots a_i$, where $r_i = 0$ and $x^4 \mid g_i(x)$ if $i-m$ is even, and $r_i = 1$ and $x^3 \mid g_i(x)$ if $i-m$ is odd.
    \item[(IV)] For $i > m$, $a_1 \cdots a_{i-1}(a_i b_i) a_i = \left((-1)^{\frac{i+m-3}{2}} r_i \ve_1 + h_i(\ve_1)\right) a_1 \cdots a_i$, where $h_i(x) \in K[x]$ and $x^2 \mid h_i(x)$, with $r_i = 0$ if $i-m$ is even, and $r_i = 1$ if $i-m$ is odd.
\end{enumerate}
\end{lemma}

\begin{proof}
In part (I), we use induction on $i$ in $a_1 \cdots a_{i-1} \ve_i$. When $i = 2$, from (\ref{eq:jacobian_ideal_b}) and (\ref{eq:jacobian_ideal_ve}), we have
\[
a_1 \ve_2 = -\ve_1 a_1 - t_1(a_1 b_1)^2 a_1 = -\ve_1 a_1 - t_1k_1^2 \ve_1^4 a_1,
\]
let $f_1(x) = -t_1k_1^2 x^3$, then the lemma holds for $i = 2$. Assume for $i \leq s$ there exists a $f_i(x)$ such that the lemma holds. When $i = s+1$, from (\ref{eq:jacobian_ideal_b}), we get
\begin{equation}\label{eq:FdDimEqsI1}
\begin{aligned}
a_1 \cdots a_s \ve_{s+1} &= -a_1 \cdots \ve_s a_s - t_s a_1 \cdots (a_s b_s)^2 a_s.
\end{aligned}
\end{equation}
Considering $a_1 \cdots a_{s-1} (a_s b_s) a_s$, from (\ref{eq:jacobian_ideal_ve}), we get
\begin{equation}\label{eq:FdDimEqsII}
\begin{aligned}
&a_1 \cdots (a_s b_s) a_s = -a_1 \cdots a_{s-1} (b_{s-1} a_{s-1}) a_s - k_s a_1 \cdots a_{s-1} \ve_s^2 a_s \\
&= (-1)^{s-1} a_1 b_1 a_1 \cdots a_{s-1} a_s + \sum_{i=2}^{s} (-1)^{s-i+1} k_i a_1 \cdots a_{i-1} \ve_i^2 a_i \cdots a_s \\
&= \sum_{i=1}^{s} (-1)^{s-i+1} k_i a_1 \cdots a_{i-1} \ve_i^2 a_i \cdots a_s \\
&= \left(\sum_{i=1}^{s} (-1)^{s-i+1} k_i \left((-1)^{i-1} + f_i(\ve_1)\right)^2\right) \ve_1^2 a_1 \cdots a_s.
\end{aligned}
\end{equation}
Substituting equation (\ref{eq:FdDimEqsII}) back into the right-hand side of equation (\ref{eq:FdDimEqsI1}), we get
\begin{equation}\label{eq:FdDimEqsI2}
\begin{aligned}
&a_1 \cdots a_s \ve_{s+1} = -a_1 \cdots \ve_s a_s - t_s \left(\sum_{i=1}^{s} (-1)^{s-i+1} k_i \left((-1)^{i-1} + f_i(\ve_1)\right)^2\right)^2 \ve_1^4 a_1 \cdots a_s \\
&= (-1)^s \ve_1 a_1 \cdots a_s - \left(t_s \left(\sum_{i=1}^{s} (-1)^{s-i+1} k_i \left((-1)^{i-1} + f_i(\ve_1)\right)^2\right)^2 \ve_1^3 + f_s(\ve_1)\right) \ve_1 a_1 \cdots a_s.
\end{aligned}
\end{equation}
Let $f_{s+1}(x) = - t_s \left(\sum_{i=1}^{s} (-1)^{s-i+1} k_i \left((-1)^{i-1} - f_i(x)\right)^2\right)^2 x_1^3 + f_s(x)$, then $x^3 \mid f_{s+1}(x)$. Thus, (I) is proved. (II) follows directly from (I) and equation (\ref{eq:FdDimEqsII}).

To prove (III), we first prove case for $i=m+1$. By applying (\ref{eq:jacobian_ideal_ve}) and lemma \ref{lemma:jacobian_finite_relations} part (II), we obtain:
\begin{equation}\label{eq:FdDimEqsIII1}
\begin{aligned}
& a_1 \cdots a_m (a_{m+1} b_{m+1})^2 a_{m+1} = a_1 \cdots a_{m+1} \left(t_m (a_m b_m)^2 + \ve_m\right)^2 a_m a_{m+1} \\
&= t_m^2 a_1 \cdots a_{m-1} (a_m b_m)^4 a_m a_{m+1} + 2 t_m a_1 \cdots a_{m-1} (a_m b_m)^2 \ve_m a_m a_{m+1} + \\
&\quad a_1 \cdots a_{m-1} \ve_m^2 a_m a_{m+1} \\
&= t_m^2\left(\left(\sum_{j=1}^i (-1)^{i-j+1} k_j \right) \ve_1^2 + g_i(\ve_1)\right)^4 a_1 \cdots a_{m+1} +\\ 
&\quad  2 t_m \left(\left(\sum_{j=1}^i (-1)^{i-j+1} k_j \right) \ve_1^2 + g_i(\ve_1)\right)^2\left((-1)^{m-1}+f_{m}(\ve_1) \right)\ve_1 a_1 \cdots a_{m+1}  +\\
&\quad  \left((-1)^{m-1}+f_{m}(\ve_1) \right)^2 \ve_1^2 a_1 \cdots a_{m+1}.
\end{aligned}
\end{equation}
Thus, there exists $r_{m+1}=1$ and 
\begin{equation}
\begin{aligned}
&g_{m+1}(x) = t_m^2\left(\left(\sum_{j=1}^m(-1)^{m-j+1}k_j\right)\ve_1^2+g_m(x) \right)^4 +\\ 
&\quad 2 t_m \left(\left(\sum_{j=1}^m(-1)^{m-j+1}k_j\right)\ve_1^2+g_m(x) \right)^2\left( (-1)^m+f_m(x)\right) x + 2\cdot (-1)^m f_m(x)x^2 + f_m(x)^2 x^2.
\end{aligned}
\end{equation}
 such that $x^4 \mid g_{m+1}(x)$, and (III) holds for $i = m+1$. Assume for $ i\in[m+1,s-1]$, there exist $g_i(x)$ and $r_i$ such that (III) holds. When $i = s$, we have
\begin{equation}\label{eq:FdDimEqsIII2}
\begin{aligned}
& a_1 \cdots a_{s-1} (a_s b_s)^2 a_s = a_1 \cdots a_{s-1} \left(t_{s-1} (a_{s-1} b_{s-1})^2 + b_{s-2} a_{s-2}\right)^2 a_{s-1} a_s \\
&= t_{s-1}^2 a_1 \cdots a_{s-2} (a_{s-1} b_{s-1})^4 a_{s-1} a_s + 2 t_{s-1} a_1 \cdots a_{s-2} (a_{s-1} b_{s-1})^2 (b_{s-2} a_{s-2}) a_{s-1} a_s + \\
&\quad + a_1 \cdots (a_{s-2} b_{s-2})^2 a_{s-2} a_{s-1} a_s \\
&= t_{s-1}^2 \left(r_{s-1} \ve_1^2 + g_{s-1}(\ve_1)\right)^2 a_1 \cdots a_s + 2 t_{s-1} \left(r_{s-1} \ve_1^2 + g_{s-1}(\ve_1)\right) a_1 \cdots a_{s-3} \cdot \\
&\quad \cdot (a_{s-2} b_{s-2}) a_{s-2} a_{s-1} a_s + \left(r_{s-2} \ve_1^2 + g_{s-2}(\ve_1)\right) a_1 \cdots a_s.
\end{aligned}
\end{equation}
Considering $a_1 \cdots a_{s-3} (a_{s-2} b_{s-2})$, we have
\begin{equation}\label{eq:FdDimEqsIV1}
\begin{aligned}
& a_1 \cdots a_{s-3} (a_{s-2} b_{s-2}) = a_1 \cdots a_{s-4} \left(-t_{s-3} (a_{s-3} b_{s-3})^2 a_{s-3} - b_{s-4} a_{s-4} a_{s-3}\right) \\
&= (-1) a_1 \cdots (a_{s-4} b_{s-4}) a_{s-4} a_{s-3} - t_{s-3} a_1 \cdots a_{s-4} (a_{s-3} b_{s-3})^2 a_{s-2}.
\end{aligned}
\end{equation}
When $s-2-m$ is even, we have
\begin{equation}\label{eq:FdDimEqsIV2}
\begin{aligned}
& a_1 \cdots a_{s-3} (a_{s-2} b_{s-2}) = (-1) a_1 \cdots (a_{s-4} b_{s-4}) a_{s-4} a_{s-3} - t_{s-3} a_1 \cdots a_{s-4} (a_{s-3} b_{s-3})^2 a_{s-2} \\
&= \sum_{p=1}^{\frac{s-m-2}{2}} \left((-1)^p t_{s-(2p+1)}\right) a_1 \cdots a_{s-2p-2} (a_{s-(2p+1)} b_{s-(2p+1)})^2 a_{s-2p} \cdots a_{s-2} + \\
&\quad + (-1)^{\frac{s-m-2}{2}} a_1 \cdots (a_m b_m) a_m \cdots a_{s-2} \\
&= \sum_{p=1}^{\frac{s-m-2}{2}} \left((-1)^p t_{s-(2p+1)}\right) \left(r_{s-(2p+1)} \ve_1^2 + g_{s-(2p+1)}(\ve_1)\right)a_1 \cdots a_{s-2} +\\ 
&\quad\quad\quad (-1)^{\frac{s-m-2}{2}}\left(\left(\sum_{j=1}^m(-1)^{m-j+1}k_j\right)\ve_1^2+g_m(\ve_1)\right) a_1 \cdots a_{s-2}.
\end{aligned}
\end{equation}
When $s-2-m$ is odd, we have
\begin{equation}\label{eq:FdDimEqsIV3}
\begin{aligned}
&a_1 \cdots a_{s-3} (a_{s-2} b_{s-2}) = (-1) a_1 \cdots (a_{s-4} b_{s-4}) a_{s-4} a_{s-3} - t_{s-3} a_1 \cdots a_{s-4} (a_{s-3} b_{s-3})^2 a_{s-2} \\
&= \sum_{p=2}^{\frac{s-m-1}{2}} \left((-1)^p t_{s-(2p-1)}\right) a_1 \cdots a_{s-2p} (a_{s-(2p-1)} b_{s-(2p-1)})^2 a_{s-2p} \cdots a_{s-2} + \\
&\quad + (-1)^p a_1 \cdots (a_{m+1} b_{m+1}) a_{m+1} \cdots a_{s-2} \\
&= \left(\sum_{p=2}^{\frac{s-m}{2}} \left((-1)^p t_{s-(2p-1)}\right) \left(r_{s-(2p-1)} \ve_1^2 + g_{s-(2p-1)}(\ve_1)\right)\right) a_1 \cdots a_s + (-1)^p a_1 \cdots \\
&\quad \cdots a_{m-1} \left(-t_m (a_m b_m)^2 a_m - \ve_m a_m\right) a_{m+1} \cdots a_{s-2} \\
&= \left(\sum_{p=2}^{\frac{s-m}{2}} \left((-1)^p t_{s-(2p-1)}\right) \left(r_{s-(2p-1)} \ve_1^2 + g_{s-(2p-1)}(\ve_1)\right)\right) a_1 \cdots a_{s-2} + (-1)^p \cdot \\
&\quad \cdot \left(-t_m \left(\left(\sum_{j=1}^m(-1)^{m-j+1}k_j\right)\ve_1^2+g_m(\ve_1)\right) - f_m(\ve_1)\right) a_1 \cdots a_{s-2} + \\ 
&\quad (-1)^{\frac{s+m-3}{2}} \ve_1 a_1 \cdots a_{s-2}.
\end{aligned}
\end{equation}

Therefore, we can obtain that when $s-m$ is even, there exists $h_{s-2}(x) \in K[x]$ such that $a_1 \cdots a_{s-3} (a_{s-2} b_{s-2}) = h_{s-2}(\ve_1) a_1 \cdots a_{s-2}$ and $x^2 \mid h_{s-2}(x)$. When $s-m$ is odd, there exists $h_{s-2}(x) \in K[x]$ such that $a_1 \cdots a_{s-3} (a_{s-2} b_{s-2}) = (-1)^m \ve_1 a_1 \cdots a_{s-2} + h_{s-2}(\ve_1) a_1 \cdots a_{s-2}$ and  $x^2 \mid h_{s-2}(x)$. Substituting this result back into equation (\ref{eq:FdDimEqsIII2}), we get that when $s-m$ is even,
\begin{equation}\label{eq:FdDimEqsIII3}
\begin{aligned}
&a_1 \cdots a_{s-1} (a_s b_s)^2 a_s = t_{s-1}^2 \left(r_{s-1} \ve_1^2 + g_{s-1}(\ve_1)\right)^2 a_1 \cdots a_s + \\
&\quad  2 t_{s-1} \left(r_{s-1} \ve_1^2 + g_{s-1}(\ve_1)\right) h_{s-2}(\ve_1)a_1 \cdots a_s +  \left(r_{s-2} \ve_1^2 + g_{s-2}(\ve_1)\right) \ve_1^2 a_1 \cdots a_s,
\end{aligned}
\end{equation}
Let $g_s(x) = t_{s-1}^2 \left(r_{s-1} x^2 + g_{s-1}(x)\right)^2  + 2 t_{s-1} \left(r_{s-1} x^2 + g_{s-1}(x)\right) h_{s-2}(x)+ \left(r_{s-2} x^2 + g_{s-2}(x)\right) x^2$, we have $x^4\mid g_s(x)$. When $s-m$ is odd,
\begin{equation}\label{eq:FdDimEqsIII4}
\begin{aligned}
a_1 \cdots a_{s-1} (a_s b_s)^2 a_s &= t_{s-1}^2 \left(r_{s-1} \ve_1^2 + g_{s-1}(\ve_1)\right)^2 a_1 \cdots a_s + 2 t_{s-1} \left(r_{s-1} \ve_1^2 + g_{s-1}(\ve_1)\right)\cdot \\
&\cdot \left((-1)^m \ve_1 + h_{s-2}(\ve_1)\right) + r_{s-1} \ve_1^2 a_1 \cdots a_s,
\end{aligned}
\end{equation}
Let $g_s(x) = t_{s-1}^2 \left(r_{s-1} x^2 + g_{s-1}(x)\right)^2  + 2 t_{s-1} \left(r_{s-1} x^2 + g_{s-1}(x)\right)\left( (-1)^m x+h_{s-2}(x) \right) +  g_{s-2}(x) x^2$, then we have $x^3\mid g_s(x)$. Thus, (III) is proved. (IV) can be directly obtained from (\ref{eq:FdDimEqsIV2}) and (\ref{eq:FdDimEqsIV3}).
\end{proof}

\begin{lemma}\label{lem:Eqs_in_Empty_Loop_set}
  In the Jacobian algebra $\widehat{\Jac}(\overline{\bA}_n^{\emptyset},W_n^{\emptyset})$, we have:
  \[a_1\cdots a_k(a_{k+1}b_{k+1})=f_k(a_1b_1)(a_1b_1)a_1\cdots a_k.\]
  where $f_k(x) \in K[x]$ satisfies $f_k(0)=0$ if $k \equiv 1 \pmod{2}$, and $f_k(0)=(-1)^{\frac{i-2}{2}}$ if $k \equiv 0 \pmod{2}$.
\end{lemma}

\begin{proof}
  We use induction on $k$ in $a_1\cdots a_k(a_{k+1}b_{k+1})$. When $k=1$, we have:
  \[a_1(a_2b_2)=(-t_1)(a_1b_1)^2a_1=(-t_1a_1b_1)(a_1b_1)a_1.\]
  When $k=2$, we have:
  \[a_1a_2(a_3b_3)=-(a_1b_1)a_1a_2-t_2a_1(a_2b_2)^2a_2.\]
  Assume the lemma holds for $k=N$. When $k=N+1$, we have:
  \begin{equation}
   \begin{aligned}
      & a_1\cdots a_N(a_{N+1}b_{N+1})\\
      =&-a_1\cdots a_{N-2}(a_{N-1}b_{N-1})a_{N-1}a_N-t_N a_1\cdots a_{N-1}(a_N b_N)^2a_N\\
      =&-f_{N-2}(a_1b_1)(a_1b_1)a_1\cdots a_N(a_{N+1}b_{N+1})-t_N f_{N-1}^2(a_1b_1)^2(a_1b_1)^2a_1\cdots a_N\\
      =&[-f_{N-2}(a_1b_1)-t_N f_{N-1}^2(a_1b_1)^2(a_1b_1)](a_1b_1)a_1\cdots a_N
   \end{aligned}   
  \end{equation}
 Thus, we have proved this lemma.
\end{proof}

From the above relations in $\widehat{\Jac}(\overline{\mathbb A}_n^{\emptyset},W_n^{\emptyset})$, we have the following zero-relations in this Jacobian algebra.

\begin{corollary}\label{cor:0RelEvenInLoopEmpty}
In the Jacobian algebra $\widehat{\Jac}(\overline{\bA}_n^{\emptyset},W_n^{\emptyset})$, if $n \equiv 0 \pmod{2}$ and $(\sum_{s=1}^{\frac{n-1}{2}} (-1)^s t_{2s-1}) \neq 0$, then
\[
(a_1b_1)^2 a_1 \cdots a_{n-1} = 0.
\]
\end{corollary}

\begin{proof}
    By repeatedly applying (\ref{eq:jacobian_ideal_b}), we can deduce:
    \begin{equation}
        \begin{aligned}
            &(a_1b_1)^2 a_1\cdots a_{n-1}\\
           =&\frac{-1}{t_1}a_1a_2b_2a_1\cdots a_{n-1}\\
           =&(-1)^{\frac{n-1}{2}}\frac{-1}{t_1}a_1\cdots a_{n-2}b_{n-2}a_{n-2}a_{n-1}+(\sum_{i=1}^{\frac{n-1}{2}}(-1)^{i+1}\frac{t_{2i+1}}{t_1}f_{2i}(a_1b_1)^2)(a_1b_1)^2a_1\cdots a_{n-1}
        \end{aligned}
    \end{equation}
    By Lemma \ref{lem:Eqs_in_Empty_Loop_set}, we have 
      \begin{equation}
        \begin{aligned}
            &(a_1b_1)^2 a_1\cdots a_{n-1}\\
           =&(\sum_{s=2}^{\frac{n-1}{2}} (-1)^s)\frac{t_{2s-1}}{t_1}(a_1b_1)^2a_1\cdots a_{n-1}+F(a_1b_1)(a_1b_1)^2a_1\cdots a_{n-1}
        \end{aligned}
    \end{equation}  
where $x\mid F(x)$. If $K=(\sum_{s=1}^{\frac{n-1}{2}} (-1)^s t_{2s-1}) \neq 0$, then we obtain:
      \begin{equation}
        \begin{aligned}
            &(a_1b_1)^2 a_1\cdots a_{n-1}\\
           =&\frac{t_1F(a_1b_1)}{K}(a_1b_1)^2a_1\cdots a_{n-1}
        \end{aligned}
    \end{equation}
where $\dfrac{t_1F(a_1b_1)}{K}\neq 0$. Therefore, for any $N \in \mathbb{Z}_{\geq 0}$, we have $(a_1b_1)^2 a_1 \cdots a_{n-1} \in \mathfrak{m}^N$, thus $(a_1b_1)^2 a_1 \cdots a_{n-1}$ converges to $0$ in $\widehat{\Jac}(\overline{\bA}_n^{\emptyset},W_n^{\emptyset})$. This completes the proof.
\end{proof}

\begin{corollary}\label{cor:0RelOddInLoopEmpty}
In the Jacobian algebra $\widehat{\Jac}(\overline{\bA}_n^{\emptyset},W_n^{\emptyset})$, if $n \equiv 1 \pmod{2}$, then
\[
(a_1b_1) a_1 \cdots a_{n-1} = 0.
\]
\end{corollary}

\begin{proof}
        By repeatedly applying (\ref{eq:jacobian_ideal_b}), we can deduce:
        \begin{equation}
        \begin{aligned}
            &(a_1b_1) a_1\cdots a_{n-1}\\
           =&-a_1a_2a_3b_3a_4\cdots a_{n-1}-t_2a_1(a_2b_2)^2a_2\cdots a_{n-1}\\
           =&(-1)^{\frac{n-1}{2}}a_1\cdots a_{n-2}b_{n-2}a_{n-2}a_{n-1}+(\sum_{i=1}^{\frac{n-1}{2}}(-1)^{i+1}\frac{t_{2i+1}}{t_1}f_{2i+1}(a_1b_1)^2)(a_1b_1)^2a_1\cdots a_{n-1}\\
           =&(\sum_{i=1}^{\frac{n-3}{2}}(-1)^{i+1}\frac{t_{2i+1}}{t_1}f_{2i+1}(a_1b_1)^2)(a_1b_1)^2a_1\cdots a_{n-1}\\
        \end{aligned}
    \end{equation}
    By lemma \ref{lem:Eqs_in_Empty_Loop_set}, we have that $x\mid f_{2i+1}(x)$, therefore, for any $N \in \mathbb{Z}_{\geq 0}$, we have $(a_1b_1) a_1 \cdots a_{n-1} \in \mathfrak{m}^N$, thus $(a_1b_1) a_1 \cdots a_{n-1}$ converges to $0$ in $\widehat{\Jac}(\overline{\bA}_n^{\emptyset},W_n^{\emptyset})$. This completes the proof.
\end{proof}

\begin{corollary}\label{cor:0RelEven}
In the Jacobian algebra $\widehat{\Jac}(\overline{\bA}_n^{[1,m]},W_n^{[1,m]})$, if $n \equiv m \pmod{2}$ and $(\sum_{i=1}^m (-1)^i k_i + \sum_{s=1}^{\frac{n-m-2}{2}} (-1)^s t_{m+2s-1}) \neq 0$, then
\[
\ve_1^2 a_1 \cdots a_{n-1} = 0.
\]
\end{corollary}

\begin{proof}
By repeatedly applying (\ref{eq:jacobian_ideal_b}), we can deduce
\begin{equation}\label{eq:0Rel1}
\begin{aligned}
\ve_1^2 a_1 \cdots a_{n-1} &= \frac{-1}{k_1} a_1 b_1 a_1 \cdots a_{n-1} \\
&= \frac{1}{k_1} a_1 a_2 b_2 a_2 \cdots a_{n-1} + \frac{k_2}{k_1} a_1 \ve_2^2 a_2 \cdots a_{n-1} \\
&= \frac{(-1)^m}{k_1} a_1 \cdots a_m b_m a_m \cdots a_{n-1} + \sum_{i=2}^m \frac{k_i}{k_1} (-1)^i a_1 \cdots a_{i-1} \ve_i^2 a_i \cdots a_{n-1}.
\end{aligned}
\end{equation}
By Lemma \ref{lemma:jacobian_finite_relations}(I) and (\ref{eq:jacobian_ideal_b}), we can further derive from (\ref{eq:0Rel1})
\begin{equation}\label{eq:0Rel2}
\begin{aligned}
&\ve_1^2 a_1 \cdots a_{n-1} = \frac{(-1)^m}{k_1} a_1 \cdots a_m b_m a_m a_{m+1} \cdots a_{n-1} + \sum_{i=2}^m \frac{k_i}{k_1} (-1)^i \ve_1^2 a_1 \cdots a_{n-1} + \\
&\quad \sum_{i=2}^m \frac{k_i}{k_1} (-1)^i \left(2f_i(\ve_1)+f_i(\ve_1)^2\right) \ve_1^2 a_1 \cdots a_{n-1}  \\
&= \frac{(-1)^{m+1}}{k_1} a_1 \cdots a_{m+2} (b_{m+2} a_{m+2} a_{m+3}) a_{m+4} \cdots a_{n-1} + \\
&\quad \frac{(-1)^{m+1} t_{m+1}}{k_1} a_1 \cdots a_m (a_{m+1} b_{m+1})^2 \cdot  a_{m+1} \cdots a_{n-1} + \\ 
&\quad \sum_{i=2}^m \frac{k_i}{k_1} (-1)^i \ve_1^2 a_1 \cdots a_{n-1} + \sum_{i=2}^m \frac{k_i}{k_1} (-1)^i \left(2f_i(\ve_1)+f_i(\ve_1)^2\right) \ve_1^2 a_1 \cdots a_{n-1} \\
&=\quad \sum_{s=1}^{\frac{n-m}{2}} (-1)^s \frac{t_{m+2s-1}}{k_1} a_1 \cdots a_{m+2s-2} (a_{m+2s-1} b_{m+2s-1})^2 a_{m+2s+1} \cdots a_{n-1} + \\
&\quad \sum_{i=2}^m \frac{k_i}{k_1} (-1)^i \ve_1^2 a_1 \cdots a_{n-1} + \sum_{i=2}^m \frac{k_i}{k_1} (-1)^i \left(2f_i(\ve_1)+f_i(\ve_1)^2\right) \ve_1^2 a_1 \cdots a_{n-1} \\
\end{aligned}
\end{equation}
By Lemma \ref{lemma:jacobian_finite_relations} (III), let $G(x) = \sum_{s=1}^{\frac{n-m}{2}} (-1)^s \frac{t_{m+2s-1}}{k_1} g_{m+2s-1}(x) + \sum_{i=2}^m \frac{k_i}{k_1} (-1)^i\left( 2f_i(x)+f_i(x)^2 x^2\right)$, we have $x^3 \mid G(x)$, thus we get
\begin{equation}\label{eq:0Rel3}
\begin{aligned}
\ve_1^2 a_1 \cdots a_{n-1} &= \left(\sum_{i=2}^m \frac{k_i}{k_1} (-1)^i + \sum_{s=1}^{\frac{n-m}{2}} (-1)^s \frac{t_{m+2s-1}}{k_1} r_{m+2s-1}\right) \ve_1^2 a_1 \cdots a_{n-1}\\
&+ G(\ve_1) \ve_1^2 a_1 \cdots a_{n-1}..
\end{aligned}
\end{equation}
Let $K = \left(\sum_{i=2}^m \frac{k_i}{k_1} (-1)^i + \sum_{s=1}^{\frac{n-m}{2}} (-1)^s \frac{t_{m+2s-1}}{k_1} r_{m+2s-1}\right)$, we get that when $K \neq 1$, we have
\[
\ve_1^2 a_1 \cdots a_{n-1} = \frac{G(\ve_1)}{1-K} \ve_1^2 a_1 \cdots a_{n-1}.
\]
Therefore, for any $N \in \mathbb{Z}_{\geq 0}$, we have $\ve_1^2 a_1 \cdots a_{n-1} \in \mathfrak{m}^N$, thus $\ve_1^2 a_1 \cdots a_{n-1}$ converges to $0$ in $\widehat{\Jac}(\overline{\bA}_n^{[1,m]},W_n^{[1,m]})$. This completes the proof.
\end{proof}

\begin{corollary}\label{cor:0RelOdd}
In the Jacobian algebra $\widehat{\Jac}(\overline{\bA}_n^{[1,m]},W_n^{[1,m]})$, if $n \equiv m \pmod{2}$ and $(\sum_{i=1}^m (-1)^i k_i + \sum_{s=1}^{\frac{n-m-4}{2}} (-1)^s t_{m+2s+1}) \neq 0$, then
\[
\ve_1^3 a_1 \cdots a_{n-2} = 0.
\]
\end{corollary}

\begin{proof}
By a similar process as in the proof of Corollary \ref{cor:0RelEven}, we can obtain:
\begin{equation}\label{eq:0RelOdd1}
\begin{aligned}
&\ve_1^6 a_1 \cdots a_{n-2} = \sum_{i=2}^m \frac{k_i}{k_1} (-1)^i \ve_1^6 a_1 \cdots a_{n-2} + \sum_{i=2}^m \frac{k_i}{k_1} (-1)^i \left(2f_i(\ve_1)+f_i(\ve_1)^2\right) \ve_1^4 a_1 \cdots a_{n-2} + \\
&\quad + \sum_{s=1}^{\frac{n-m-4}{2}} (-1)^s \frac{t_{m+2s+1}}{k_1} \ve_1^4 a_1 \cdots a_{m+2s-2} (a_{m+2s+1} b_{m+2s+1})^2 a_{m+2s+1} \cdots a_{n-2} + \\
&\quad + \frac{(-1)^{(n-m-2)/2}}{k_1} \ve_1^4 a_1 \cdots a_{n-3} a_{n-2} b_{n-2} a_{n-2}.
\end{aligned}
\end{equation}
and
\begin{equation}\label{eq:0RelOdd2}
\begin{aligned}
& \ve_1^4 a_1 \cdots a_{n-2} = \sum_{i=2}^m \frac{k_i}{k_1} (-1)^i \ve_1^4 a_1 \cdots a_{n-2} + \sum_{i=2}^m \frac{k_i}{k_1} (-1)^i \left(2f_i(\ve_1)+f_i(\ve_1)^2\right) \ve_1^2 a_1 \cdots a_{n-2} + \\
&\quad + \sum_{s=1}^{\frac{n-m-4}{2}} (-1)^s \frac{t_{m+2s+1}}{k_1} \ve_1^2 a_1 \cdots a_{m+2s-2} (a_{m+2s+1} b_{m+2s+1})^2 a_{m+2s+1} \cdots a_{n-2} + \\
&\quad + \frac{(-1)^{(n-m-2)/2}}{k_1} \ve_1^2 a_1 \cdots a_{n-3} a_{n-2} b_{n-2} a_{n-2}.
\end{aligned}
\end{equation}

Considering $\ve_1^2 a_1 \cdots a_{n-3} (a_{n-2} b_{n-2})^2 a_{n-2}$, by equation (\ref{eq:jacobian_ideal_b}), Lemma \ref{lemma:jacobian_finite_relations}, and corollary \ref{cor:0RelEven}, we have:
\begin{equation}\label{eq:0RelOdd3}
\begin{aligned}
\ve_1^2 a_1 \cdots a_{n-3} (a_{n-2} b_{n-2})^2 a_{n-2} &= (-1)^{m+1} \ve_1^3 a_1 \cdots a_{n-2} + f(\ve_1) a_1 \cdots a_{n-2},
\end{aligned}
\end{equation}
where $f(x)$ satisfies $x^4 \mid f(x)$. Substituting (\ref{eq:0RelOdd2}) and (\ref{eq:0RelOdd3}) into (\ref{eq:0RelOdd1}) and simplifying, we get:
\[
\ve_1^3 a_1 \cdots a_{n-2} = F(\ve_1) a_1 \cdots a_{n-2},
\]
where $F(x)$ satisfies $x^4 \mid F(x)$, hence $\ve_1^3 a_1 \cdots a_{n-2}$ converges to $0$ in $\widehat{\Jac}(\overline{\bA}_n^{[1,m]},W_n^{[1,m]})$.
\end{proof}
\subsection{Finite-dimensionality of $\widehat{\Jac}(\overline{\bA}_n^{[1,m]},W_n^{[1,m]})$}
we consider specific paths in the 2-cyclic quiver $Q$ and analyze their contributions within the Jacobian algebra $\widehat{\Jac}(\overline{\bA}_n^{[1,m]},W_n^{[1,m]})$. The following corollary formalizes the conditions under which certain paths must vanish in $\widehat{\Jac}(\overline{\bA}_n^{[1,m]},W_n^{[1,m]})$.

\begin{corollary}\label{cor:JacFin_1_to_N_1_to_N-1}
If a path $p$ in the 2-cyclic quiver $Q$ satisfies $s(p)=1$, $t(p)=n$, and $l(p)\geq n+3$, then $\overline{p}=0$ in $\widehat{\Jac}(\overline{\bA}_n^{[1,m]},W_n^{[1,m]})$. If $p$ satisfies $s(p)=1$, $t(p)=n-1$, and $l(p)\geq n+3$, then $\overline{p}=0$ in $\widehat{\Jac}(\overline{\bA}_n^{[1,m]},W_n^{[1,m]})$.
\end{corollary}

\begin{proof}
We only provide the proof for the case \(m \geq 1\), the proof for \(m < 1\) is similar. It is known that in $Q$, any path from $1$ to $n$ must contain the arrows $a_1, \dots, a_{n-1}$. Therefore, we can deduce that if a path $p$ in $Q$ satisfies $s(p)=1$, $t(p)=n$, and $l(p)\geq n+3$, then there exist some cycles $c_i$ in $Q$ such that $p$ is of the form $a_1 \cdots a_{s(c_1)-1} c_1 a_{s(c_1)} \cdots a_{s(c_i)-1} c_i a_{s(c_i)} \cdots a_n$. Furthermore, from the structure of $Q$, we can deduce that $c_i$ are cycles of the form $a_{s(c_i)} a_{s(c_i)+1} \cdots \ve_k^u \cdots a_{k'} b_{k'} \cdots b_{s(c_i)+1} b_{s(c_i)}$, where $u\geq 0$. Since $\sum_{i \in I} l(c_i) \geq 4$, by Lemma \ref{lemma:jacobian_finite_relations}, we know that $p = f(\ve_1) a_1 \cdots a_{n-1}$ in $\widehat{\Jac}(\overline{\bA}_n^{[1,m]},W_n^{[1,m]})$, where $x^2 \mid f(x)$. Therefore, by Corollary \ref{cor:0RelEven}, we know that $p=0$. The proof of the second part of the corollary is similar.
\end{proof}

Extending our analysis of path conditions within the 2-cyclic quiver $Q$, we now check the paths that start from a vertex $d > 1$ and end at vertex $n$. The following lemma establishes that any such path greater than a specified length can be expressed as a linear combination of shorter paths within the Jacobian algebra $\widehat{\Jac}(\overline{\bA}_n^{[1,m]},W_n^{[1,m]})$. 

From the potential $W_n^{[1,m]}$ (equation \ref{eq:potential_form}) and the relations (equation \ref{eq:jacobian_ideal_ve}, \ref{eq:jacobian_ideal_a}, \ref{eq:jacobian_ideal_b}) induced by the differential $\partial_\alpha(W_n^{[1,m]})$, it follows that $\wh{\Jac}(\overline{\bA}_n^{[1,m]},W_n^{[1,m]})/\la e_1,…,e_{d-1}\ra\simeq \wh{\Jac}(\overline{\bA}_{n-d+1}^{[1,m-d+1]},W_{n-d+1}^{[1,m-d+1]})$ and $\wh{\Jac}(\overline{\bA}_n^{[1,m]},W_n^{[1,m]})/\la e_{u+1},…,e_{n}\ra\simeq \wh{\Jac}(\overline{\bA}_{n-u-1}^{[1,m]},W_{n-u-1}^{[1,m]})$.

\begin{lemma}\label{lem:JacFin_d_to_N}
If a path $p$ in the 2-cyclic quiver $Q$ satisfies $s(p)=d>1$, $t(p)=n$, and $l(p)>n+d+2$, then $p$ in $\widehat{\Jac}(\overline{\bA}_n^{[1,m]},W_n^{[1,m]})$ can be expressed linearly in terms of paths of length not greater than $n+d+2$ in $\widehat{\Jac}(\overline{\bA}_n^{[1,m]},W_n^{[1,m]})$.
\end{lemma}

\begin{proof}
Let $s(p)=m>1$. Consider the algebra homomorphism $\pi: \widehat{\Jac}(\overline{\bA}_n^{[1,m]},W_n^{[1,m]}) \rightarrow \widehat{\Jac}(\overline{\bA}_n^{[1,m]},W_n^{[1,m]}) / \langle e_1, \dots, e_{d-1} \rangle$. By corollary \ref{cor:JacFin_1_to_N_1_to_N-1} and the $[1,m]$-finite dimensional condition, we know that $\pi(p) = 0$ in $\widehat{\Jac}(\overline{\bA}_n^{[1,m]},W_n^{[1,m]}) / \langle e_1, \dots, e_{d-1} \rangle$. Therefore, in $\widehat{\Jac}(\overline{\bA}_n^{[1,m]},W_n^{[1,m]})$, either $p = 0$, or there exist $p_1, p_1'$ such that $s(p_1)=d$, $t(p_1)=s(p_1')=d_1<d$, $t(p_1')=n$, and $p = p_1 p_1'$. If $p = 0$, the lemma is proved; if $p \neq 0$, and if $l(p_1 p_1') \leq n+d+2$, the lemma is proved; if $l(p_1 p_1') > n+d+2$, we can choose appropriate $p_1'$ and $d_1 < d$ such that $l(p_1') > n+d_1+3$, $l(p_1) = d-d_1$. By repeatedly applying the same process to $p_1'$, we obtain that $p = 0$ or $p = p_1 \cdots p_h$, where $p_h$ satisfies $s(p_h)=1$, $p_1 \cdots p_{h-1}$ satisfies $l(p_1 \cdots p_{h-1})=d$. If $l(p_1) > n+3$, by corollary \ref{cor:JacFin_1_to_N_1_to_N-1}, we obtain that $p=0$; if $l(p_1) \leq n+3$, then $l(p_1 \cdots p_h) \leq n+d+2$. Thus, the lemma is proved.
\end{proof}

\begin{lemma}\label{lem:JacFin_1_to_u}
If a path $p$ in the 2-cyclic quiver $Q$ satisfies $s(p)=1$, $t(p)=u<n$, and $l(p)>2n-u+2$, then $\overline{p}$ in $\widehat{\Jac}(\overline{\bA}_n^{[1,m]},W_n^{[1,m]})$ can be expressed linearly in terms of paths of length not greater than $n+u+2$ in $\widehat{\Jac}(\overline{\bA}_n^{[1,m]},W_n^{[1,m]})$.
\end{lemma}

\begin{proof}
Let $u' \geq u$ and satisfy $u' = u$ when $u \equiv m \pmod{2}$, $u' = u+1$ when $u+1 \equiv m \pmod{2}$. Consider the algebra homomorphism $\pi: \widehat{\Jac}(\overline{\bA}_n^{[1,m]},W_n^{[1,m]}) \rightarrow \widehat{\Jac}(\overline{\bA}_n^{[1,m]},W_n^{[1,m]}) / \langle e_{u'}, \dots, e_n \rangle$. By Corollary \ref{cor:JacFin_1_to_N_1_to_N-1} and the $[1,m]$-finite dimensional condition, we know that $\pi(p) = 0$ in $\widehat{\Jac}(\overline{\bA}_n^{[1,m]},W_n^{[1,m]}) / \langle e_{u'}, \dots, e_n \rangle$. Therefore, in $\widehat{\Jac}(\overline{\bA}_n^{[1,m]},W_n^{[1,m]})$, either $p = 0$, or there exist $p_1', p_1$ such that $t(p_1)=u$, $t(p_1')=s(p_1)=u_1>u$, $s(p_1')=1$, and $p = p_1' p_1$. If $p = 0$, the lemma is proved; if $p \neq 0$, and if $l(p_1' p_1) \leq 2n-u+2$, the lemma is proved; if $l(p_1' p_1) > 2n-u+2$, we can choose appropriate $p_1'$ and $u_1 > u$ such that $l(p_1') > 2n-u_1+2$, $l(p_1)=u_1-u$. By repeatedly applying the same process to $p_1'$, we eventually obtain $p = 0$ or $p = p_h \cdots p_1$, where $p_h$ satisfies $t(p_h)=n$ or $n-1$, $p_{h-1} \cdots p_1$ satisfies $l(p_{h-1} \cdots p_1)=n-u$. If $l(p_h) > n+2$, by Corollary \ref{cor:JacFin_1_to_N_1_to_N-1}, we know $p=0$; if $l(p_1) \leq n+2$, then $l(p_1 \cdots p_h) \leq 2n-u+2$. Thus, the lemma is proved.
\end{proof}

By lemma \ref{lem:JacFin_d_to_N} and lemma \ref{lem:JacFin_1_to_u}, we further obtain:
\begin{corollary}\label{cor:JacFin_d_to_u}
If a path $p$ in the 2-cyclic quiver $Q$ satisfies $s(p)=d>1$, $t(p)=u<N$, and $l(p)>N+u+d$, then $\overline{p}$ in $\widehat{\Jac}(\overline{\bA}_n^{[1,m]},W_n^{[1,m]})$ can be expressed linearly in terms of paths of length less than $N+u+d$ in $\widehat{\Jac}(\overline{\bA}_n^{[1,m]},W_n^{[1,m]})$.
\end{corollary}

Let $P_{\leq 2N}$ denote the set of all paths in $Q$ of length not greater than $2N$. From the above results, we can obtain:
\begin{corollary}\label{cor:JacFin}
All elements in the completed Jacobian algebra $\widehat{\Jac}(\overline{\bA}_n^{[1,m]},W_n^{[1,m]})$ corresponding to the 2-cyclic quiver with potential $(\overline{\bA}_n^{[1,m]},W_n^{[1,m]})$ can be expressed linearly in terms of the elements in $P_{\leq 2N}$ in $\widehat{\Jac}(\overline{\bA}_n^{[1,m]},W_n^{[1,m]})$. Furthermore, $P_{\leq 2N}$ is a finite set, $\widehat{\Jac}(\overline{\bA}_n^{[1,m]},W_n^{[1,m]})$ is a finite-dimensional algebra.
\end{corollary}

\section{Applications}\label{sec: Application} 
\subsection{Jacobian-finite Potentials for $\bZ_3$-covers of the 2-cyclic Quivers with Potentials}\label{sec:Jacobian-finite_Potentials_for_Z3-covers_of_the_2-cyclic_Quivers_with_Potentials}

In this section, we define a Jacobian-finite potential for a class of quivers without loops and 2-cycles. These quivers with potentials are induced from 2-cyclic quiver with potential $(\overline{\bA_n}^{[1,m]},W_n^m)$ satisfied the $[1,m]$-finite condition through the covering theory. Next, we provide an introduction to $G$-covering of quivers with potential for any finite group $G$.

\subsubsection{$G$-Coverings of quivers with potentials} 

In this section, we always assume that $G$ is a finite group. If a map $\vf:Q\rightarrow Q'$ satisfies $s(\vf(a))=\vf(s(a))$ and $t(\vf(a))=\vf(t(a))$ for any $a\in Q^1$, then we call $\vf$ a homomorphism from $Q$ to $Q'$. If $\vf$ is an isomorphism, then we call $\vf$ an isomorphism from $Q$ to $Q'$. Let $\vf$ be an automorphism on $Q$. Clearly, for any cycle $c$ in $Q$, we have $\vf(c)=c$. We call $\vf$ an automorphism of the quiver with potential $(Q,W)$ if for any summand $\lambda p$ in $W$, $\lambda\vf(p)$ is also a summand in $W$. Let $G$ be an automorphism group on $(Q,W)$. If for any $\vf\neq 1 \in G$, $\vf(a)\neq a$ for all $a\in Q^0$, then we say that $G$ acts admissibly on $(Q,W)$. Clearly, when $G$ is \textbf{admissible} on $(Q,W)$, we have $|G|\mid|Q^0|$ and $|G|\mid|Q^1|$. Furthermore, let $W=\sum \lambda_p p$. When $G$ is admissible, we can construct a new quiver with potential $(Q_G,W_G)$ from $(Q,W)$, where:
\[
(Q_G)^0=\{Gi|i\in Q^0\}, \quad (Q_G)^1=\{Ga|a\in Q^1 \},
\]
\[
W_G=\sum |G|\lambda_p Gp.
\]

\begin{lemma}\cite{PS18}
When $G$ acts admissibly on $(Q,W)$, the elements of $G$ induce automorphisms on $\wh \Jac(Q,W)$. Hence, $G$ is also an automorphism group on $\wh \Jac(Q,W)$.
\end{lemma}

Next, we introduce covering functors.

\begin{definition}\label{def:CoveringFunc}
Let $\cA$ and $\cB$ be two $K$-categories, and let $G$ be an automorphism group on $\cA$. We call a functor $F:\cA\rightarrow\cB$ a $G$-precovering functor if for any $a,b\in \cA$, there are natural equivalences:
\[
\bigoplus_{g\in G}\Hom_\cA(a,gb)\simeq\Hom_{\cB}(Fa,Fb),
\]
\[
\bigoplus_{g\in G}\Hom_\cA(ga,b)\simeq\Hom_{\cB}(Fa,Fb).
\]
If the functor $F$ is dense, then we call $F$ a $G$-covering functor.
\end{definition}

Clearly, when $G$ acts admissibly on $(Q,W)$, the map $\pi:Q\rightarrow Q_G, \pi(a)=Ga$ is a $G$-covering functor. Starting from $\pi$, we can also induce the $G$-covering functor $\pi:\wh{KQ}\rightarrow \wh{KQ_G}$. Furthermore, Proposition 3.1 and Lemma 3.2 in \cite{PS18} proved:

\begin{proposition}\label{prop:JacCover}\cite{PS18}
The $G$-covering functor $\pi:Q\rightarrow Q_G$ induces the $G$-covering functor:
\[
\pi:\wh \Jac(Q,W)\rightarrow \wh \Jac(Q_G,W_G),
\]
\end{proposition}

From the $G$-covering functor $\pi:\wh \Jac(Q,W)\rightarrow \wh \Jac(Q_G,W_G)$, we obtain that $e_{Ga}\wh \Jac(Q_G,W_G)e_{Gb}$ is finite-dimensional if and only if $\oplus_{b_i\in Gb} (e_{a}\wh \Jac(Q_G,W_G)e_{b_i})$ is finite-dimensional. If $Q$ is a finite quiver, we have:

\begin{corollary}\label{cor:Cover_of_Potential}
  When $G$ acts admissibly on $(Q,W)$, the Jacobian algebra $\wh \Jac(Q,W)$ is finite-dimensional if and only if $\wh \Jac(Q_G,W_G)$ is finite-dimensional.
\end{corollary}

\subsubsection{Jacobian-finite quiver with potential $(C_{3}(\overline{\bA}_n^I), C_3(W_n^I))$}

We start to define a class of 2-acyclic quiver, denoted as $C_{3}(\bA_n^I)$, which is a $\bZ_3$-cover of the 2-cyclic quiver $\overline{\bA}_n^I$. The vertex set $C_{3}(\overline{\bA}_n^I)_0$ of $C_{3}(\overline{\bA}_n^I)$ is defined as the set $\{i_j \mid i \in \overline{\bA}_n^I, 1 \leq j \leq 3\}$ with cardinality $3n$. Each subset $\{i_1, i_2, i_3\}$ corresponds to the fiber of $i \in \overline{\bA}_n^I$ in $C_{3}(\overline{\bA}_n^I)_0$ under the $\bZ_3$-action.

We define the following four types of cycles on the set $\{i_1, i_2, i_3\} \cup \{(i+1)_1, (i+1)_2, (i+1)_3\}$:
\[
\begin{tikzcd}
L_i^{(1)}: & i_1 \arrow[r, "{\ve_{i,1}}"] & i_2 \arrow[r, "{\ve_{i,2}}"] & i_3 \arrow[r, "{\ve_{i,3}}"] & i_1 \\
L_i^{(2)}: & i_1 \arrow[r, "{\ve_{i,1}}"] & i_3 \arrow[r, "{\ve_{i,3}}"] & i_2 \arrow[r, "{\ve_{i,2}}"] & i_1 \\
C_{i,i+1}^{(1)}: & i_1 \arrow[r, "{a_{i,1}}"] & (i+1)_3 \arrow[r, "{b_{i,1}}"] & i_3 \arrow[r, "{a_{i,3}}"] & (i+1)_2 \arrow[r, "{b_{i,3}}"] & i_2 \arrow[r, "{a_{i,2}}"] & (i+1)_1 \arrow[r, "{b_{i,2}}"] & i_1 \\
C_{i,i-1}^{(2)}: & i_1 \arrow[r, "{a_{i,1}}"] & (i+1)_2 \arrow[r, "{b_{i,2}}"] & i_2 \arrow[r, "{a_{i,2}}"] & (i+1)_3 \arrow[r, "{b_{i,3}}"] & i_3 \arrow[r, "{a_{i,3}}"] & (i+1)_1 \arrow[r, "{b_{i,1}}"] & i_1
\end{tikzcd}
\]

We then define the arrows in $C_{3}(\bA_n^I)$ recursively:
\begin{itemize}
    \item[(1)] If $1 \in I$, all the arrows connecting the vertices in $\{1_1, 1_2, 1_3\}$ form a cycle of type $L_1^{(1)}$. If $1 \notin I$, the vertices in $\{1_1, 1_2, 1_3\}$ are not directly connected by arrows in $C_{3}(\bA_n^I)$, and all the arrows connecting $\{1_1, 1_2, 1_3\}$ and $\{2_1, 2_2, 2_3\}$ form a cycle of type $C_{1,2}^{(1)}$.
    \item[(2)] If $i \in I$ and $i \neq 1$, all the arrows connecting the vertices in $\{i_1, i_2, i_3\}$ form a cycle of type $L_i^{(1)}$ or $L_i^{(2)}$. Additionally, all the arrows connecting $\{i_1, i_2, i_3\}$ and $\{(i+1)_1, (i+1)_2, (i+1)_3\}$ form a cycle of type $C_{i,i+1}^{(1)}$ or $C_{i,i+1}^{(2)}$. Both types are determined by the cycle type between $\{(i-1)_1, (i-1)_2, (i-1)_3\}$ and $\{i_1, i_2, i_3\}$: if it is $C_{i-1,i}^{(1)}$, then $L_i^{(1)}$ and $C_{i,i+1}^{(1)}$; if it is $C_{i-1,1}^{(2)}$, then $L_i^{(2)}$ and $C_{i,i+1}^{(2)}$.
    \item[(3)] If $i \notin I$, all the arrows between $\{i_1, i_2, i_3\}$ and $\{(i+1)_1, (i+1)_2, (i+1)_3\}$ form a cycle of type $C_{i,i+1}^{(1)}$ or $C_{i,i+1}^{(2)}$, determined by the cycle type between $\{(i-1)_1, (i-1)_2, (i-1)_3\}$ and $\{i_1, i_2, i_3\}$: if it is $C_{i-1,i}^{(1)}$, then $C_{i,i+1}^{(2)}$; if it is $C_{i-1,i}^{(2)}$, then $C_{i,i+1}^{(1)}$.
\end{itemize}

\begin{example}
The quiver $C_3(\overline{\bA}_3^{[1,1]})$ is of the form:
\begin{figure}[H]
    \centering
    \includegraphics[width=0.4\linewidth]{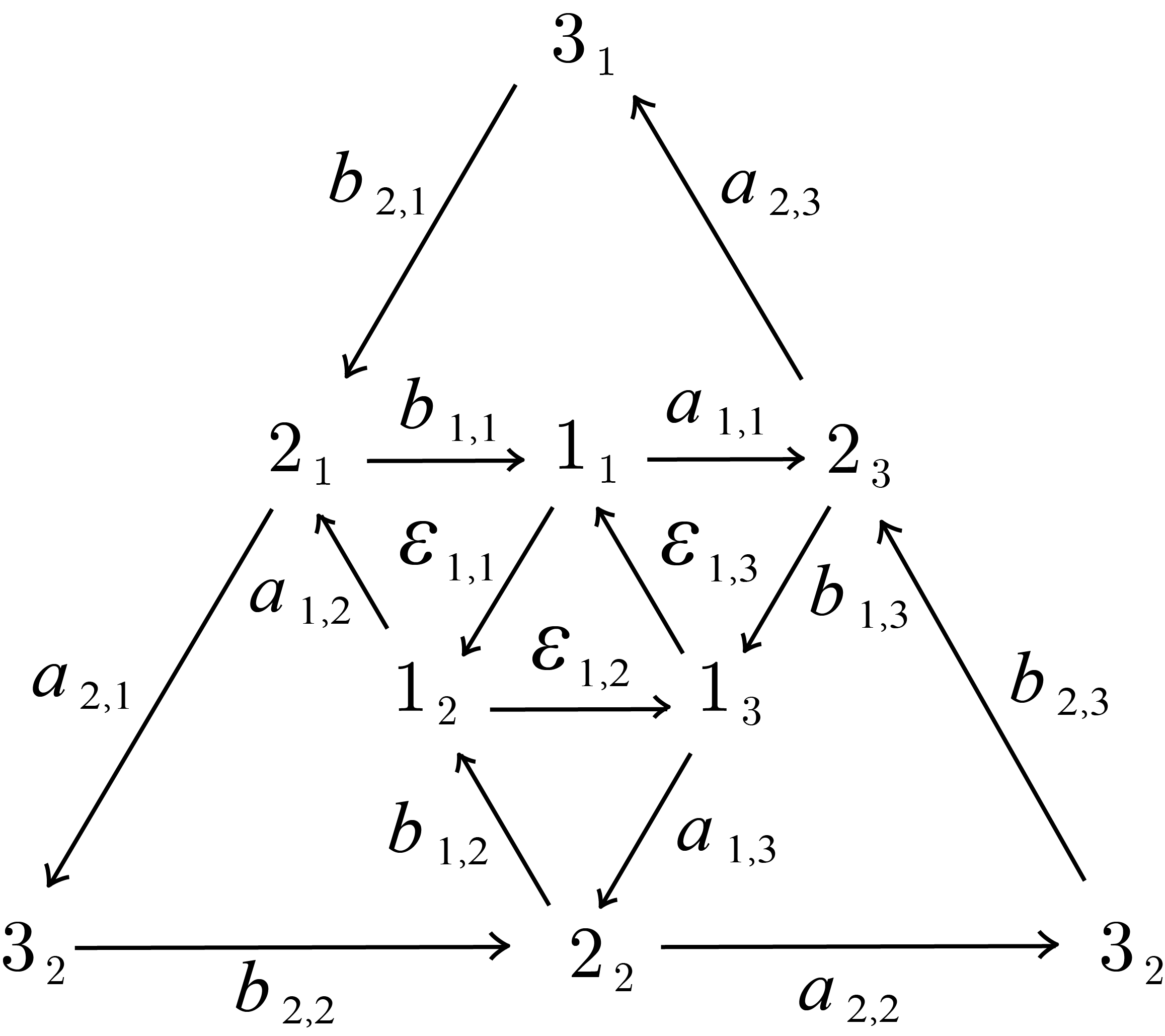}
    \caption{$C(\overline{\bA}_3^{[1,1]})$}
    \label{fig:Z3011cover}
\end{figure}
\end{example}

Consider the $\bZ_3$-cover $C_3(\overline{\bA}_n^{[1,m]})$ of 2-cyclic quiver $\overline{\bA}_n^{[1,m]}$, there exists a potentials $C_3(W_n^{[1,m]})$ on $C_3(\overline{\bA}_n^{[1,m]})$ of the form:
\begin{equation}\label{eq:Z3cover_of_potential}
\begin{aligned}
&C_3(W_n^{[1,m]})=\sum_{i=1}^m(k_i\ve_{i,1}\ve_{i,2}\ve_{i,3}+t_ia_{i,1}b_{i,3}a_{i,3}b_{i,2}a_{i,2}b_{i,1})\\
&+\sum_{i=1}^m\sum_{j=1}^3 (\ve_{i,j}a_{i,j+1}b_{i,j+1}+\ve_{i,j}a_{i-1,j+1}b_{i-1,j+1})+\\
&+\sum_{\substack{ i\geq m,\\i\equiv m\pmod{2}}}\sum_{j=1}^3 (a_{i,j}a_{i+1,j-1}b_{i+1,j}b_{i,j}+t_ia_{i,1}b_{i,3}a_{i,3}b_{i,2}a_{i,2}b_{i,1})\\
&+\sum_{\substack{ i\geq m,\\i\equiv m+1\pmod{2}}}\sum_{j=1}^3 (a_{i,j}a_{i+1,j+1}b_{i+1,j+1}a_{i+1,j}+t_ia_{i,1}b_{i,2}a_{i,2}b_{i,3}a_{i,3}b_{i,1})
\end{aligned}
\end{equation}
Let $g$ be an automorphism of $C_3(\overline{\bA}_n^{[1,m]})$ satisfied $g(i_j)=i_{j+1}, g(\ve_{i,j})=\ve_{i,j+1}, g(a_{i,j})=a_{i,j+1}, g(b_{i,j})=b_{i,j+1}$, it is not hard to see that $G=\la g^i \ra\simeq \bZ_3$ and $G$ acts admissibly on $(C_3(\overline{\bA}_n^{[1,m]}),C_3(W_n^{[1,m]}))$, furthermore, we have $(GC_3(\overline{\bA}_n^{[1,m]}),GC_3(W_n^{[1,m]}))=(\overline{\bA}_n^{[1,m]},W_n^{[1,m]})$. By corollary \ref{cor:JacFin}, we know that $W_n^{[1,m]}$ is a Jacobian-finite potential if it satisfied the $[1,m]$-finite dimensional condition. Thus, by corollary \ref{cor:Cover_of_Potential}, we obtain:
\begin{theorem}
  The quiver with potential $(C_3(\overline{\bA}_n^{[1,m]}),C_3(W_n^{[1,m]}))$ defined as above is Jacobian-finite if for any $1 \leq i' \leq m$ and $1 \leq s' \leq \frac{n-m-2}{2}$, $k_i,t_i$ satisfied the 
  equation \ref{eq:finite_dimensional_condition}.
\begin{equation}
\begin{aligned}
\left( \sum_{i=i'}^{m} (-1)^i k_i + \sum_{s=1}^{s'} (-1)^s t_{m+2s-1} \right) \neq 0
\end{aligned}
\label{eq:finite_dimensional_condition}
\end{equation}
\end{theorem}
Thus, we obtained a Jacobian-finite potential for the quiver $C_3(\overline{\bA}_n^{[1,m]})$. By theorem \ref{theorem:Main_theorem_DWZ}, we further obtain:
\begin{corollary}
  Given a quiver with potential $(Q,W)$ that is obtained through a finite sequence of mutations from $(C_3(\overline{\bA}_n^{[1,m]}),C_3(W_n^{[1,m]}))$, and assuming $W_n^{[1,m]}$ satisfies the $[1,m]$-finite dimensional condition, it follows that $(Q,W)$ is Jacobian finite.
\end{corollary}

\begin{remark}
  When $\overline{\bA}_n^{[1,m]}$ is exactly the quiver $\overline{\bA_2}^{\emptyset}$, $\overline{\bA_2}^{[1,2]}$ and $\overline{\bA_3}^{\emptyset}$, the potential $C_3(W_2^{\empty})$, $C_3(W_2^{[1,2]})$ and $C_3(W_3^{\emptyset})$ corresponds to the Labardini-potential of the triangulations on the once-puncture disk with 6 marked points, the sphere with 4 punctures and the sphere with 5 punctures respectively, and it has been proven that they are all non-degenerated potentials in \cite{LF1}, as a general case, we also conjecture that $(C(\bA_n^{I},W_n^{I}))$ is a non-degenerated potential.
\end{remark}

\begin{conjecture}
For any $m,n\in\bZ_n$, the potential $C(W_n^I)$ is non-degenerated.
\end{conjecture}

\subsection{Categorification of Paquette Schiffler's Generalized Cluster Algebras}\label{sec: Categorification_of_PS_Generalized_Cluster_Algebras}

An orbifold is a triple \((S, M, O)\), where \(S\) is a Riemann surface, \(M\) is a finite subset of \(S\), and \(O\) is a finite subset of \(S\) equipped with a valued map \(v:O \rightarrow \mathbb{Z}_{\geq 1}\). Inspired by the works of Fomin, Shapiro and Thurston \cite{FST08, FT18}, Paquette and Schiffler defined the generalized cluster algebra for an orbifold in \cite{PS18}, which generalizes the cluster algebra from a triangulated surface. In this paper, we do not recall the specific definition of Paquette-Schiffler's generalized cluster algebras but directly state the specific generalized cluster algebras for each cases. For more details of Paquette-Schiffler's generalized cluster algebras, readers can refer to \cite{PS18}. In the following discussion of this section, for any vector $a=(a_1,...,a_n)^T \in \mathbb{Z}^n$, we denote the monomial $x_1^{a_1}\cdots x_n^{a_n}$ as $x^a$.

We now provide the definition of the 2-cyclic Caldero-Chapoton formula for the module \(M \in \mathrm{mod}~\widehat{\Jac}(\overline{\mathbb{A}}^{[1,m]}_n, W_n^{[1,m]})\) and the definition of $\tau$-rigid modules and support $\tau$-tilting pair.

\begin{definition}\label{def:tau_tlit}\cite{AIR14}
Let $A$ be a finite-dimensional $K$-algebra, an $A$-module $M$ is called $\tau$-rigid if $\Hom_A(M, \tau M) = 0$.
\end{definition}

\begin{definition}\label{def:2cyclicCCmap}
For any $M \in \mathrm{mod}~\widehat{\Jac}(\overline{\mathbb{A}}^{[1,m]}_n, W_n^{[1,m]})$, the following map is defined:
\begin{equation}\label{eq:2cyclicCCmap}
CC(M) = \left( \sum_{e \in \mathbb{N}^n} \chi(\mathrm{Gr}_e(M))) \right) x^{-g_M},
\end{equation}
where $\Gr_e(M)$ is the quiver Grassmannian of $M$ with respect to the dimension vector $e$. This map is referred as the 2-cyclic Caldero-Chapoton map (or 2-cyclic CC-map).
\end{definition}

In the following two subsections, to provide a precise description of Paquette-Schiffler's generalized cluster algebras for each case, we will first list all the generators of Paquette-Schiffler's generalized cluster algebras. By using the 2-cyclic Caldero-Chapoton formula and the Jacobian-finite potential defined in Section \ref{section:jacobian_finite_2cyclic}, we will prove that for each case of the 2-cyclic quiver, the indecomposable \(\tau\)-rigid modules of the corresponding Jacobian algebras correspond to the non-initial cluster variables of the respective Paquette-Schiffler's generalized cluster algebras.

\subsubsection{Categorification of Paquette Schiffler's Generalized Cluster Algebras of $\overline{\bA_2}^{\emptyset}$}

In this section, we consider the Paquette-Schiffler's generalized cluster algebras $\cA(\overline{\bA_2}^{\emptyset})$ of $\overline{\bA_2}^{\emptyset}$, it is corresponding to the Paquette-Schiffler's generalized cluster algebras of $(D^2,\{ m_1,m_2, p \},v)$, where $v(p)=3$. This is a subalgebra of $\bZ[x_1^\pm, x_2^\pm]$. All the generators of $\cA(\overline{\bA_2}^{\emptyset})$ are of the forms:
\[
x_1,x_2,\dfrac{2x_2}{x_1},\dfrac{2x_1}{x_2},\\ 
\dfrac{6}{x_1},\dfrac{6}{x_2}
\]

\begin{remark}
There exists a triangulation on $(D^2,\{ m_1,m_2, p \},v)$ such that its associated quiver is $\overline{\bA_2}^{\emptyset}$.
\begin{figure}[H]
    \centering
    \includegraphics[width=0.3\linewidth]{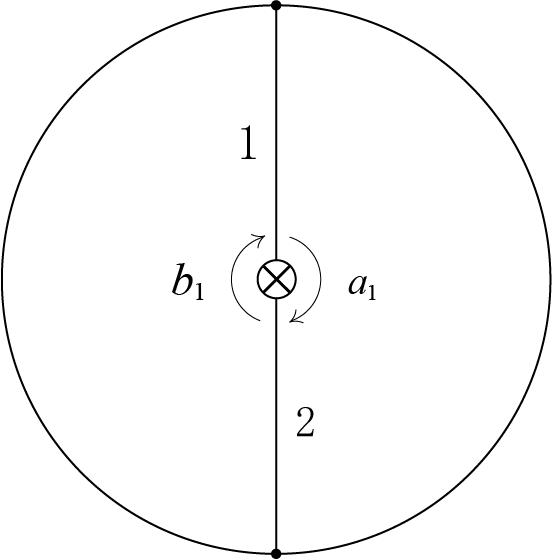}
    \caption{$(D^2,\{ m_1,m_2, p \},v)$}
    \label{fig:D2}
\end{figure}
\end{remark}

From Theorem \ref{cor:JacFin}, it is known that $(\overline{\bA_2}^{\emptyset}, (a_1b_1)^3)$ is a Jacobian-finite quiver with potential. The Auslander-Reiten quiver of $\modu~\wh{\Jac}(\overline{\bA_2}^{\emptyset},(a_1b_1)^3)$ is as follows:
\begin{figure}[H]
    \centering
    \includegraphics[width=0.6\linewidth]{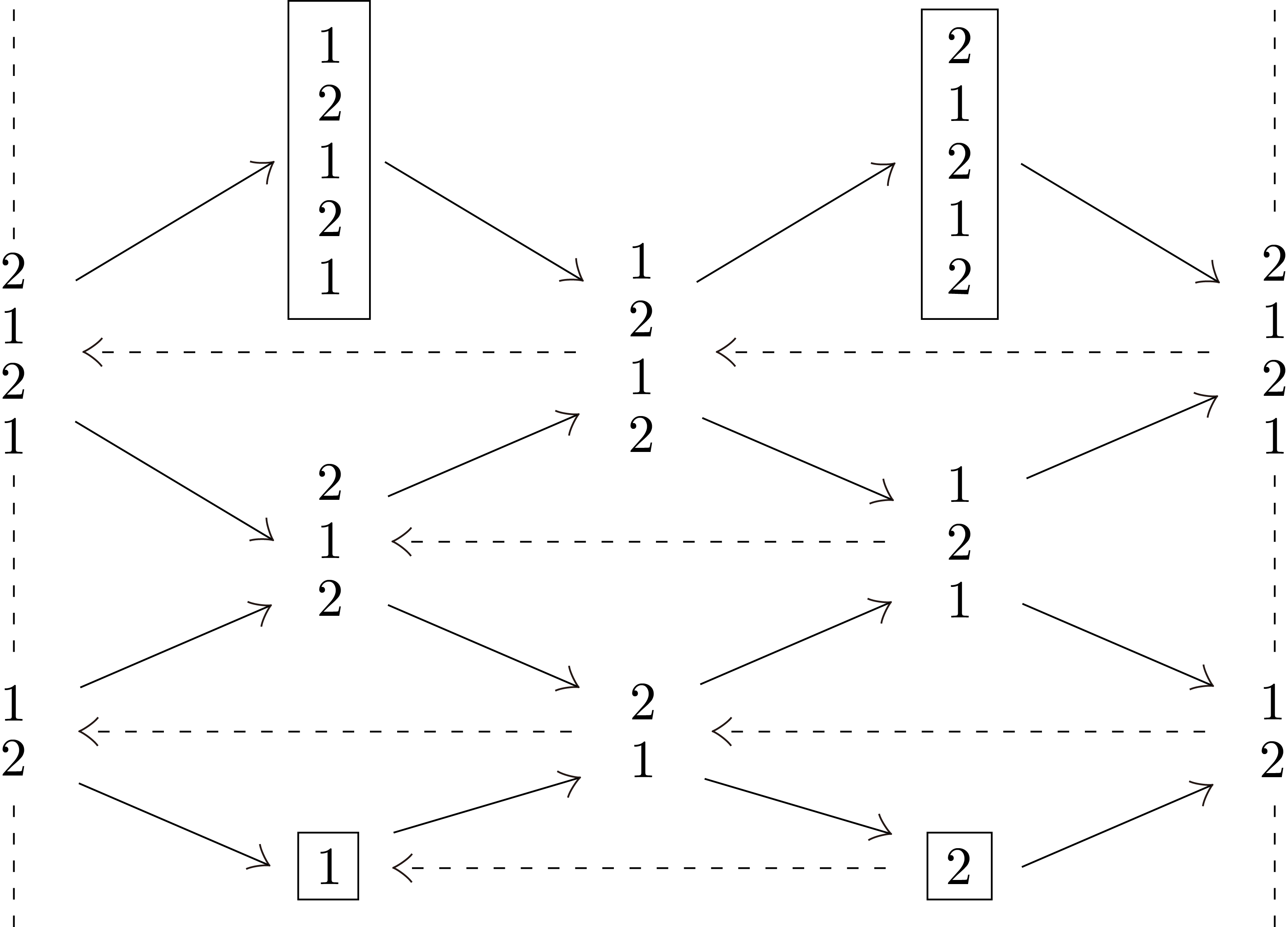}
    \caption{Auslander-Reiten quiver of $\modu~\wh{\Jac}(\overline{\bA_2}^{\emptyset},(a_1b_1)^3)$}
    \label{fig:ARQuiver_A00}
\end{figure}

The indecomposable modules enclosed by the square in the Auslander-Reiten quiver are the indecomposable $\tau$-rigid modules of $\modu~\wh{\Jac}(\overline{\bA_2}^{\emptyset},(a_1b_1)^3)$. We denote them as follows:

\begin{equation}\label{eq:TauRigidA11}
\begin{aligned}
S_1=1,~&
S_2=2,~&
P_1=~\begin{matrix}
 1\\
 2\\
 1\\
 2\\
 1
\end{matrix},~&
P_2=~\begin{matrix}
 2\\
 1\\
 2\\
 1\\
 2
\end{matrix}.
\end{aligned} 
\end{equation}
The $g$-vector of each indecomposable $\tau$-rigid module is $g_{S_1}=[1,-1], g_{S_2}=[-1,1], g_{P_1}=[1,0], g_{P_2}=[0,1]$. moreover, we have:
\begin{equation}\label{eq:CCTauRigidA11}
\hspace{-1cm}
\begin{array}{rl}
CC(S_1) & = (1+\chi(\Gr_{[1,0]}(S_1)))\dfrac{x_2}{x_1}=\dfrac{2x_2}{x_1},\\[10pt]
CC(S_2) & = (1+\chi(\Gr_{[0,1]}(S_2)))\dfrac{x_1}{x_2}=\dfrac{2x_1}{x_2},\\[10pt]
CC(P_1) & = (1+\chi(\Gr_{[1,0]}(P_1))+\chi(\Gr_{[1,1]}(P_1))+\chi(\Gr_{[2,1]}(P_1)) \\
        & \hspace{3cm}+\chi(\Gr_{[2,2]}(P_1))+\chi(\Gr_{[3,2]}(P_1)))\dfrac{1}{x_1}=\dfrac{6}{x_1},\\[10pt]
CC(P_2) & = (1+\chi(\Gr_{[0,1]}(P_2))+\chi(\Gr_{[1,1]}(P_2))+\chi(\Gr_{[1,2]}(P_2)) \\
        & \hspace{3cm}+\chi(\Gr_{[2,2]}(P_2))+\chi(\Gr_{[2,3]}(P_2)))\dfrac{1}{x_2}=\dfrac{6}{x_2}.
\end{array}
\end{equation}
Thus, we obtain the following theorem:

\begin{theorem}\label{thm:Categorify_A00_GCA}
The 2-cyclic CC-map $CC(-): \modu~\wh \Jac(\overline{\bA_2}^{\emptyset}, (a_1b_1)^3) \rightarrow \bZ[x_1^{\pm},x_2^{\pm}]$ provides a bijection between the indecomposable $\tau$-rigid modules in $\modu~\wh \Jac(\overline{\bA_2}^{\emptyset}, (a_1b_1)^3)$ and the non-initial cluster variables in $\cA(\overline{\bA_2}^{\emptyset})$.
\end{theorem}

\subsubsection{Categorification of Paquette Schiffler's Generalized Cluster Algebras of $\overline{\bA_2}^{[1,2]}$}

In this section, we consider the generalized cluster algebras $\cA(\overline{\bA_2}^{[1,2]})$ of $\overline{\bA_2}^{[1,2]}$, which corresponds to the Paquette Schiffler's Generalized Cluster Algebras of $(S^2,\{ p_1,p_2 \}, \{ o\},v)$, where $v(p_1)=1,v(p_2)=3$. This is also a subalgebra of $\bZ[x_1^\pm,x_2^\pm]$, the generators of $\cA(\overline{\bA_2}^{[1,2]})$ are:
\[
x_1,x_2,\dfrac{3x_2}{x_1},\dfrac{3x_1}{x_2},\\ 
\dfrac{9}{x_1},\frac{9}{x_2}
\]

\begin{remark}
There exists a triangulation on $(S^2,\{ p_1,p_2 \}, \{ o\},v)$ whose associated quiver is $\overline{\bA_2}^{[1,2]}$
\begin{figure}[H]
    \centering
    \includegraphics[width=0.3\linewidth]{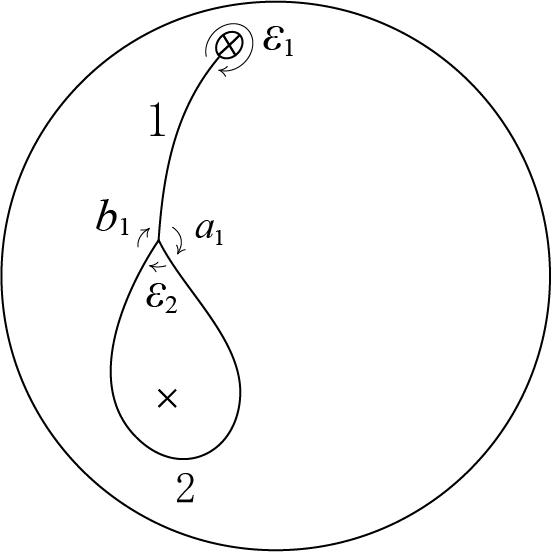}
    \caption{$(S^2,\{ p_1,p_2 \}, \{ o\},v)$}
    \label{fig:ARQuiver_A00}
\end{figure}

\end{remark}

According to the theorem \ref{cor:JacFin}, we have known that $(\overline{\bA_2}^{[1,2]}, W^{[1,2]}_2=2\ve_1^3+3\ve_1a_1b_1+3\ve_2b_1a_1+\ve_2^3)$ is a Jacobian-finite quiver with potential. The connected components containing all the indecomposable $\tau$-rigid modules $E_1,E_2,P_1,P_2$ in the Auslander-Reiten quiver of $\modu~\wh{\Jac}(\overline{\bA_2}^{[1,2]}, W^{[1,2]}_2)$ is shown below:
\begin{figure}[H]
    \centering
    \includegraphics[width=1\linewidth]{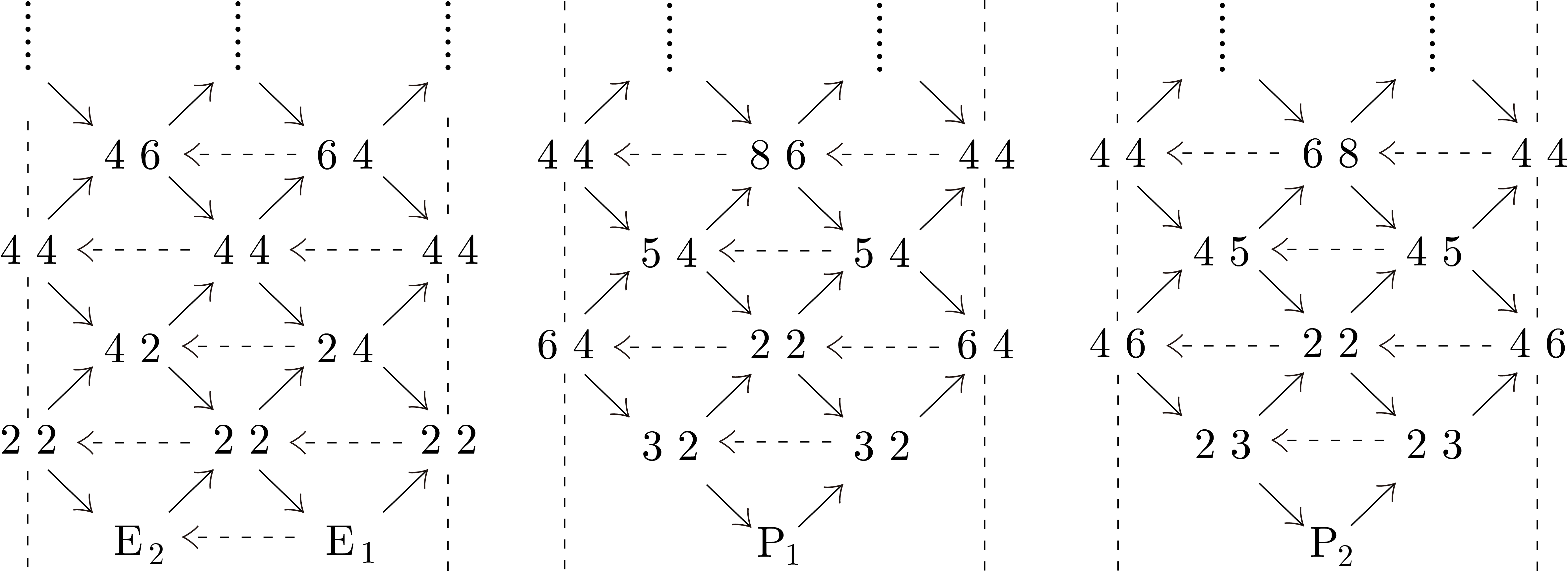}
    \caption{three connected components of Auslander-Reiten quiver of $\modu~\wh{\Jac}(\overline{\bA_2}^{[1,2]},W^{[1,2]}_2)$}
    \label{fig:ARQuiver_A00}
\end{figure}

The indecomposable $\tau$-rigid modules are specifically of the forms:

\begin{equation}\label{eq:TauRigidA11}
\begin{aligned}
E_1=~\begin{matrix}
 1\\
 1
\end{matrix}~,~&
E_2=~\begin{matrix}
 2\\
 2
\end{matrix}~,~&
P_1=~\begin{matrix}
 &  & 1 & \\ 
 & 1 &  & 2\\ 
 1&  & 2 & \\ 
 & 1 &  & 
\end{matrix}~,~&
P_2=~\begin{matrix}
 &  & 2 & \\ 
 & 2 &  & 1\\ 
 2&  & 1 & \\ 
 & 2 &  & 
\end{matrix}.
\end{aligned} 
\end{equation}
The $g$-vector of each indecomposable $\tau$-rigid modules is $g_{E_1}=[1,-1],g_{E_2}=[-1,1],g_{P_1}=[1,0],g_{P_2}=[0,1]$. For each vector $e\in \bZ^2$, it can be determined that $\Gr_e(\star)\leq 1, \star=\{E_1,E_2,P_1,P_2\}$, thus, by counting all submodules of $E_1,E_2,P_1,P_2$, we obtain:
\begin{equation}\label{eq:CCTauRigidA11}
\begin{aligned}
CC(E_1)&=\dfrac{3x_2}{x_1},
CC(E_2)&=\dfrac{3x_1}{x_2},
CC(P_1)&=\dfrac{9}{x_1},
CC(P_2)&=\dfrac{9}{x_2}.
\end{aligned} 
\end{equation}
Thus, we have:

\begin{theorem}\label{thm:Categorify_A11_GCA}
The 2-cyclic CC-map $CC(-): \modu~\wh \Jac(\overline{\bA_2}^{[1,2]}, W_2^{[1,2]}) \rightarrow \bZ[x_1^{\pm},x_2^{\pm}]$ provides a bijection between the indecomposable $\tau$-rigid modules in $\modu~\wh \Jac(\overline{\bA_2}^{[1,2]}, W_2^{[1,2]})$ and the non-initial cluster variables in $\cA(\overline{\bA_2}^{[1,2]})$.
\end{theorem}

\subsubsection{Categorification of Paquette Schiffler's Generalized Cluster Algebras of $\overline{\bA_3}^{\emptyset}$}

In this section, we consider the generalized cluster algebras $\cA(\overline{\bA_3}^{\emptyset})$ of $\overline{\bA_3}^{\emptyset}$, which corresponds to the Paquette Schiffler's Generalized Cluster Algebras of $(S^2,\{ p_1,p_2,p_3 \},v)$, where $m(p_1)=m(p_3)=3,m(p_2)=1$. This is also a subalgebra of $\bZ[x_1^\pm,x_2^\pm,x_3^\pm]$, the generators of $\cA(\overline{\bA_3}^{\emptyset})$ are:
\[
x_1,x_2,x_3,\dfrac{2x_2}{x_1},\dfrac{2x_1}{x_2},\frac{2x_1x_3}{x_2},\frac{6x_1}{x_3},\frac{6x_1}{x_2},\frac{6x_3}{x_2},\frac{6x_3}{x_1},\frac{8x_2}{x_1x_3},\\ 
\dfrac{12}{x_1},\frac{12}{x_3},\frac{36}{x_2}
\]

\begin{remark}
There exists a triangulation on $(S^2,\{ p_1,p_2,p_3 \},v)$ whose associated quiver is $\overline{\bA_3}^{\emptyset}$
\begin{figure}[H]
    \centering
    \includegraphics[width=0.35\linewidth]{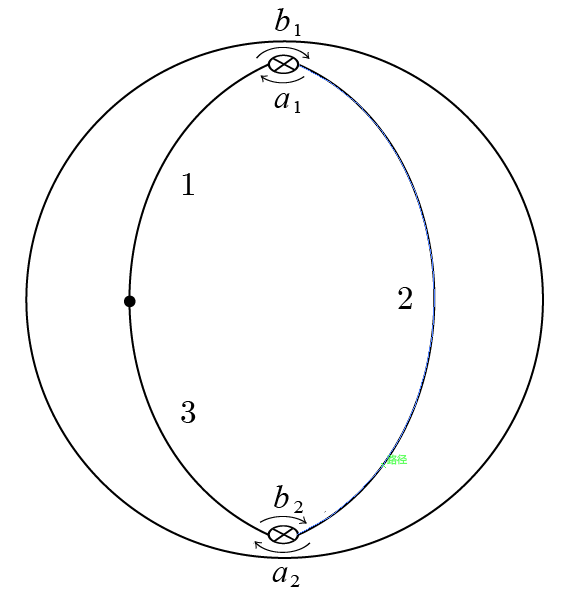}
    \caption{$(S^2,\{ p_1,p_2,p_3 \},v)$}
    \label{fig:ARQuiver_A111}
\end{figure}

\end{remark}

According to the theorem \ref{cor:JacFin}, we have known that $(\overline{\bA_3}^{\emptyset}, W^{[\emptyset]}_3)=(a_1b_1)^3+(a_2b_2)^3+3a_1a_2b_2b_1$ is a Jacobian-finite quiver with potential. Besides the simple modules, 
the non-projective indecomposable $\tau$-rigid modules are of the forms:

\begin{equation}\label{eq:Non_Proj_Tau_RigidA111}
\begin{aligned}
M_{[0,3,2]}=~\begin{matrix}
2\\
3\\
2\\
3\\
2
\end{matrix}~,~
M_{[2,3,0]}=~\begin{matrix}
2\\
1\\
2\\
1\\
2
\end{matrix}~,~M_{[3,2,0]}=~\begin{matrix}
1\\
2\\
1\\
2\\
1
\end{matrix}~,M_{[0,2,3]}=~\begin{matrix}
3\\
2\\
3\\
2\\
3
\end{matrix}~,~M_{[2,1,2]}=~\begin{matrix}
 1&  & \\ 
 & 2 & \\ 
 1&  & 3\\ 
 &  & 3 
\end{matrix}~
\end{aligned} 
\end{equation}
the projective modules $P_1,P_3$ are of the form:
\begin{equation}\label{eq:Proj_Mod_1,3}
P_1=~\begin{matrix}
 & 1 & \\ 
 & 2 & \\ 
 &  & 1\\ 
3 &  & 2\\ 
 &  & 1\\ 
 & 2 & \\ 
 & 1 & 
\end{matrix}~,~P_3=~\begin{matrix}
 & 3 & \\ 
 & 2 & \\ 
 &  & 3\\ 
1 &  & 2\\ 
 &  & 3\\ 
 & 2 & \\ 
 & 3 & 
\end{matrix}~
\end{equation}
And projective module $P_2$ is of the form:
\begin{figure}[H]
\centering
\begin{tikzpicture}[scale=1]

\node (2_1) at (0,0) {2};
\node (1_1) at (-1,-0.5) {1};
\node (2_2) at (-0.5,-1) {2};
\node (1_2) at (-0.5,-1.5) {1};
\node (p_2) at (-2.5,-1.5) {$P_2=$};
\node (2_4) at (-1,-2) {2};
\node (1_3) at (-0.5,-2.5) {1};
\node (3_1) at (1,-0.5) {3};
\node (2_3) at (0.5,-1) {2};
\node (3_3) at (0.5,-1.5) {3};
\node (2_6) at (1,-2) {2};
\node (3_2) at (0.5,-2.5) {3};
\node (2_5) at (0,-3) {2};

\draw[-] (2_2) -- (3_2);
\draw[-] (2_3) -- (1_3);

\end{tikzpicture}
\end{figure}
The $g$-vector of each indecomposable $\tau$-rigid modules is $g_{S_1}=[1,-1,0],g_{S_2}=[-1,1,-1],g_{S_3}=[0,-1,1],g_{M_{[0,3,2]}}=[-1,1,0],g_{M_{[2,3,0]}}=[0,1,-1],g_{M_{[3,2,0]}}=[1,0,-1],M_{[0,2,3]}=[-1,0,1],M_{[2,1,2]}=[1,-1,1]$. For any vector $e\in \bZ^3$, Thus, we obtain:
\begin{equation}\label{eq:CCTauRigidA111}
\begin{aligned}
&CC(S_1)=\dfrac{2x_2}{x_1}, && CC(S_2)=\dfrac{2x_1x_3}{x_2}, && CC(S_3)=\dfrac{2x_2}{x_3},\\
&CC(M_{[0,3,2]})=\dfrac{6x_1}{x_2}, && CC(M_{[2,3,0]})=\dfrac{6x_3}{x_2}, && CC(M_{[3,2,0]})=\dfrac{6x_3}{x_1},\\
&CC(M_{[0,2,3]})=\dfrac{6x_2}{x_3}, && CC(M_{[2,1,2]})=\dfrac{8x_2}{x_1x_3},\\
&CC(P_1)=\dfrac{12}{x_1}, && CC(P_2)=\dfrac{12}{x_2}, && CC(P_3)=\dfrac{36}{x_3}.
\end{aligned} 
\end{equation}
Thus, we have:

\begin{theorem}\label{thm:Categorify_A11_GCA}
The 2-cyclic CC-map $CC(-): \modu~\wh \Jac(\overline{\bA_3}^{\emptyset}, W_3^{\emptyset}) \rightarrow \bZ[x_1^{\pm},x_2^{\pm},x_3^{\pm}]$ provides a bijection between the indecomposable $\tau$-rigid modules in $\modu~\wh \Jac(\overline{\bA_3}^{\emptyset}, W_3^{\emptyset})$ and the non-initial cluster variables in $\cA(\overline{\bA_3}^{\emptyset})$.
\end{theorem}

\end{sloppypar}
\end{document}